\documentclass[final,onefignum,onetabnum]{siamonline171218}

\usepackage{lipsum}
\usepackage{amsfonts}
\usepackage{graphicx}
\usepackage{epstopdf}
\usepackage{algorithmic}
\ifpdf
  \DeclareGraphicsExtensions{.eps,.pdf,.png,.jpg}
\else
  \DeclareGraphicsExtensions{.eps}
\fi

\usepackage{mathrsfs}
\usepackage{amssymb}
\usepackage{amsmath}
\usepackage{mathtools}
\usepackage{subcaption}
\usepackage{xcolor,hyperref}
\usepackage{comment}
\usepackage[normalem]{ulem}
\usepackage{lineno}
\usepackage{amsopn}

\usepackage{enumitem}
\setlist[enumerate]{leftmargin=.5in}
\setlist[itemize]{leftmargin=.5in}

\newsiamremark{remark}{Remark}
\newsiamremark{hypothesis}{Hypothesis}
\crefname{hypothesis}{Hypothesis}{Hypotheses}
\newsiamthm{claim}{Claim}

\def\a{{\boldsymbol\alpha}}

\def\A{{\mathbf{A}}}

\def\bo{{\mathcal{O}}}

\def\d{{\partial}}

\def\e{{\epsilon}}
\def\E{{\mathcal{E}}}
\def\g{{\gamma}}

\def\J{{\mathscr{J}}}

\def\M{{\mathcal{M}}}

\def\rmc{{\mathcal{R}}}
\def\R{{\mathbb{R}}}
\def\Rn{{\mathbb{R}^{N}}}
\def\s{{\sigma}}

\def\U{{\mathbf{U}}}
\def\v{{\mathbf{v}}}
\def\bv{{\mathbf{V}}}

\def\w{{\omega}}

\def\x{{\mathbf{x}}}
\def\X{{\mathbf{X}}}

\def\z{{\mathbf{z}}}
\def\Z{{\mathbf{Z}}}

\headers{Closures for Stiff Multiscale Random Dynamics}{T.E. Maltba, H. Zhao, and D.A. Maldonado}

\title{Data-driven Closures \& Assimilation for Stiff Multiscale Random Dynamics\thanks{Submitted to the editors December 15, 2023.
\funding{This work was funded by the National Science Foundation Graduate Research Fellowship under Grant No. DGE 2146752. This material is based upon work supported by the U.S. Department of
Energy, Office of Science, Advanced Scientific Computing Research under
Contract DE-AC02-06CH11357. The work reported in this paper has been partly supported by
the U.S. Department of Energy Advanced Grid Modeling Program.}}}

\author{Tyler E. Maltba\thanks{Theoretical Division, Los Alamos National Laboratory, Los Alamos, NM 87545, USA (\email{maltba@lanl.gov}).}
\and Hongli Zhao$^\S$\thanks{Department of Statistics, University of Chicago, Chicago, IL 60637, USA (\email{honglizhaobob@uchicago.edu}).}
\and D. Adrian Maldonado\thanks{Mathematics and Computer Science Division, Argonne National Laboratory, Lemont, IL 60439, USA (\email{bzhao@anl.gov}, \email{maldonadod@anl.gov})}
}

\ifpdf
\hypersetup{
  pdftitle={Data-driven Closures \& Assimilation for Stiff Multiscale Random Dynamics},
  pdfauthor={T.E. Maltba, H. Zhao, D.A. Maldonado}
}
\fi

\begin{document}

\maketitle

\begin{abstract}
We introduce a data-driven and physics-informed framework for propagating uncertainty in stiff, multiscale random ordinary differential equations (RODEs) driven by correlated (colored) noise. Unlike systems subjected to Gaussian white noise, a deterministic equation for the joint probability density function (PDF) of RODE state variables does not exist in closed form. Moreover, such an equation would require as many phase-space variables as there are states in the RODE system. To alleviate this curse of dimensionality, we instead derive exact, albeit unclosed, reduced-order PDF (RoPDF) equations for low-dimensional observables/quantities of interest. The unclosed terms take the form of state-dependent conditional expectations, which are directly estimated from data at sparse observation times. However, for systems exhibiting stiff, multiscale dynamics, data sparsity introduces regression discrepancies that compound during RoPDF evolution. This is overcome by introducing a kinetic-like defect term to the RoPDF equation, which is learned by assimilating in sparse, low-fidelity RoPDF estimates. Two assimilation methods are considered, namely nudging and deep neural networks, which are successfully tested against Monte Carlo simulations.
\end{abstract}



\section{Introduction}

Randomness is inherent to most, if not all, complex phenomena described by ordinary differential equations (ODEs)---it enters such models in two ways (a) stochastic forcing terms accounting for internally generated or externally imposed “sub-grid” fluctuations (i.e., noise), and (b) probabilistic representations of uncertain coefficients and initial/boundary data. Owing to simplicity of implementation and parallelizablility, multilevel Monte Carlo (MC) simulations \cite{Giles} and its variants (e.g., \cite{TAVERNIERS2020}) remain as common approaches for uncertainty quantification (UQ) of random ODEs (RODEs) and stochastic differential equations (SDEs).
However, MC simulations shed little light on a system's probabilistic dynamics and are burdened by slow convergence rates, requiring significant computational resources. 

The search for efficient alternatives has led to the development of quasi-MC simulations, moment ODEs (MODEs), polynomial chaos expansions (PCEs), Mori-Zwanzig formalism (MZF), and the method of distributions (MoD), each having its strengths and weaknesses. For example, MODEs limit random inputs to be Gaussian or of small variation and are capable of providing only a few statistical moments \cite{bover_1978}, which are usually not sufficient to characterize a system's probabilistic nature. PCEs do give rise to probability density functions (PDFs) by using a finite number of uncorrelated random variables to approximate temporally varying random inputs, e.g., Karhunen-Lo\`eve (KL) expansions. However, they are inappropriate for models whose random sources have short-range correlations \cite{Wang-2013-Probability}. MZF, on the other hand, is a step in the right direction for high-dimensional RODEs by seeking non-Markovian reduced-order Langevin equations for low-dimensional observables/quantities of interest (QoIs). Such equations contain a nonlocal term that requires a closure approximation for its memory kernel. However, neither classical \cite{Chorin_2002} nor data-driven \cite{Fu_2020} approaches to kernel closures are well-suited for stiff, multiscale systems since they require integrating over most, if not all, past dynamics, i.e., the long-memory problem (see, e.g., \cite{Li_2022}). Among other restrictions, MZF also requires all noise inputs to be Gaussian and white \cite{Zhu_2021}.

The aforementioned approaches fall short because they cannot simultaneously tackle high-dimensionality, stiffness, multiple scales, and colored noise. For low-dimensional systems exhibiting these traits, the MoD, comprised of PDF and cumulative distribution function (CDF) methods, has been highly successful via the derivation and learning of closed-form deterministic partial differential equations (PDEs) for joint PDFs/CDFs of system states \cite{maltbaphd,Maltba,Maltba2022,Wang-2013-Probability}. The approach has also been adapted to quantify parametric uncertainty in hyperbolic~\cite[and the references therein]{Tartakovsky-2015-Method} and parabolic~\cite{Boso-2016-method} PDEs. Its major strength relies on random input fields being treated exactly, which is in contrast to implementations based on KL expansions~\cite{Venturi-2012-A}, meaning that, unlike PCEs, the MoD is well-suited for systems with short-range colored noise. In what follows, we limit our study to PDF methods.

The standard MoD approach is infeasible for high-dimensional RODEs since it results in a PDE with as many spatial dimensions as there are states/equations in the system. While there have been advancements in numerical integration of high-dimensional PDEs~\cite{rodgers2021adaptive}, they typically are not well-suited for complex multiscale dynamics of exceptionally large dimension.
Similar to MZF, we instead consider low-dimensional QoIs, albeit directly for their PDF dynamics instead of their Langevin ones, leading to the reduced-order MoD. More precisely, we derive exact, albeit unclosed, reduced-order PDF (RoPDF) equations for low-dimensional QoIs. Unlike MZF, the RODE noise need not be Gaussian nor white, nor does a memory kernel need approximating. 

Unclosed terms in our RoPDF equations take the form of state-dependent conditional expectations, henceforth referred to as regression functions, and are estimated from state data at discrete times. When a QoI has slow dynamics with respect to the observation time intervals, the learned regression functions produce negligible discrepancies, and solutions to the resulting RoPDF equations are accurate, assuming enough data has been injected. When the RODE is at least partially separable with respect to the QoI (in the sense of \cite{BRENNAN2018281}), parts of the regression functions are known analytically. Thus, part of the RoPDF equation is known {\emph{a priori}} and is physics-informed, which is a means for variance reduction. This was studied in \cite{BRENNAN2018281} for ODEs with random initial conditions and subsequently for It\^o SDEs in \cite{MaltbaPES}. 

For systems of stiff, multiscale RODEs, QoIs may vary rapidly over observation windows, making the available state data temporally sparse with respect to the QoI's timescale. In this setting, regression discrepancies amount to model/PDE misspecification and compound during RoPDF evolution, producing inaccurate RoPDFs. Moreover, if state data is synthetically generated via MC simulations, data availability for regression may be limited due to large computational costs associated with the numerical integration of (possibly high-dimensional) stiff RODEs and necessarily coarse time steps. To overcome these challenges, we introduce an {\emph{a priori}} unknown source term to the RoPDF equation for capturing model defects. It is inferred by post-processing the limited QoI data via fast, robust kernel density estimation (KDE) to form low-fidelity, temporally sparse RoPDF estimates, which are then assimilated into the RoPDF equation. We give a head-to-head comparison of two assimilation procedures: nudging (a.k.a., Newtonian relaxation (NR)) and deep neural networks (DNNs). The former dynamically steers the RoPDF equation solution (a.k.a, the observer) towards the RoPDF estimates via a tuned (finite) relaxation rate. The latter, however, can be interpreted as instantaneous, albeit not dynamic, relaxation.  

The paper is organized as follows. We introduce RODEs and derive deterministic PDEs for their RoPDFs in \Cref{sec:main}, and \Cref{sec:assim} discusses the data-assimilation procedures nudging and DNNs used for RoPDF inference. Details on numerics, training, and computational complexity are given in \Cref{sec:num}. Experimental results are presented in \Cref{sec:experiments} for a stiff linear system and a power grid model of transmission failures, both driven by Ornstein-Uhlenbeck (OU) noise. Concluding remarks and future directions are summarized in \Cref{sec:conclusions}.


\section{Reduced-order Method of Distributions}
\label{sec:main} Consider the RODE system
\begin{equation}  \label{ivp2}
\frac{d\x(t)}{dt} = \v\big(\x(t,\w),t;\boldsymbol\xi(t,\w)\big), \quad\quad \x(0,\w) = \x^0(\w),
\end{equation} 
to be solved on a time interval $(0, T_f]$ and holds for almost every $\omega \in \Omega$, where $(\Omega, \mathscr{F}, \mathbb{P})$ is an appropriate probability space. The solution $\x(t,\w):[0,T_f]\times \Omega \to \R^N$ is an $\mathbb R^N$-valued stochastic process with the initial state $\x^0$, defined as an $N$-dimensional random vector with joint PDF $f_{\x^0}(\X) : \Rn \to \R^+$. The phase space for \cref{ivp2} is taken as $\R^N$ for notational convenience; however, this can be altered, with little effect on the arguments below, to account for almost surely bounded processes. The given deterministic function $\v = [ v_1, \dots, v_N]^\top: \Rn\times[0,T_f] \to \Rn$, parameterized with  a set of $N_p$ random coefficients $\boldsymbol \xi(t,\w) = [\xi_1(t,\w),\dots,\xi_{N_p}(t,\w)]^\top$, satisfies conditions guaranteeing the existence of a unique pathwise solution $\x(t,\w)$ (see \cite{Han2017}). We assume without loss of generality that the random processes $\boldsymbol \xi(t,\w)$ are zero-mean and characterized by a prescribed single-time joint PDF $f_{\boldsymbol\xi}(\boldsymbol\Xi;t)$. We frequently use the shorthand $v\big(\x(t),t;\w\big)$ for the right-hand side of \cref{ivp2}, use $\mathbb E[\cdot]$ and $\langle \cdot \rangle$ interchangeably to denote the ensemble mean, and omit $\w$ in our notation when possible.

\begin{remark}
The paths of $\boldsymbol\xi$ are Lebesgue measurable, almost surely bounded, and at most H{\"o}lder continuous on $[0,T_f]$ so that \cref{ivp2} can be interpreted in the sense of Carath\'eodory. Therefore, the paths of $\x$ are continuously differentiable with derivatives that are at most H{\"o}lder continuous \cite{Han2017}.
\end{remark}

Let $\X = [X_1,\dots, X_N]^\top$ denote a phase-space variable in $\R^N$. For any fixed $t>0$, the system is (partially) characterized by the single-time joint CDF $F_\x(\X;t) \triangleq \mathbb{P}[\x(t) \le \X]$. If $F_\x(\X;t)$ is differentiable with respect to all components $X_i$, the system is equivalently characterized by the single-time joint PDF $f_\x(\X;t)$. When the state dimension $N$ is large, deriving or learning a PDF equation for $f_\x(\X;t)$ is intractable since the result would be an $N$-dimensional PDE. We instead consider a low-dimensional QoI $\z(t)\equiv \z(\x(t))\in\R^{N_\text{RO}}$, $N_{\text{RO}}<N$, where $\z:\R^N\to\R^{N_\text{RO}}$ is a continuously differentiable phase-space function guaranteeing the existence of single-time PDF $f_\z(\Z;t)$ of $\z(t)$. Here, $\Z\in\R^{N_\text{RO}}$ is a phase-space variable for $\z(t)$. We seek a deterministic PDE governing the evolution of $f_\z(\Z;t)$, referred to as the RoPDF equation. We restrict our formulation to the setting of marginal PDF equations, meaning that QoIs take the form $z(t) \triangleq x_k(t)$ for $k\in\{1,\dots,N\}$. The RoPDF equation then reduces to a one-dimensional PDE for the marginal PDF $f_{x_k}(X_k;t)$. 

We begin by defining an auxiliary functional or ``raw PDF'' \cite{Tartakovsky-2015-Method}
\begin{align} \label{rawpdf2}
	\Pi_{x_k}(X_k,t) \triangleq \delta(x_k(t)-X_k),
\end{align}
where $\delta(\cdot)$ is the Dirac delta function. We show in \cref{thm} that $\Pi_{x_k}$ weakly satisfies the random advection equation \cref{ro-clpi}. Moreover, by the Dirac delta's sifting property, for any given time $t>0$, the ensemble mean of $\Pi_k$ is $f_{x_k}$:
\begin{align} \label{eqn:ensemble-of-dirac-is-the-pdf}
	\left\langle \Pi_{x_k}\right\rangle(X_k,t) \triangleq 
	\int_\R \delta(Y-X_k)f_{x_k}(Y;t)\,dY = f_{x_k}(X_k;t).
\end{align}
Given this relationship, an exact, albeit unclosed, RoPDF equation for $f_{x_k}$ is found by stochastically homogenizing the equation for $\Pi_{x_k}$. In the setting of joint PDF equations, an analogous procedure for $f_\x$ has been the subject of several investigations \cite{Kraichnan_1961,maltbaphd,Maltba,Maltba2022}.

To proceed, we require a slight change in notation by denoting $\v(x_k(t),\x_{-k}(t),t;\w) \equiv \v(\x(t),t;\w)$, where $\x_{-k}(t) \triangleq [x_1(t),\dots,x_{k-1}(t),x_{k+1}(t),\dots,x_N(t)]^\top$, to emphasize the QoI in the velocity field. A heuristic derivation of the raw PDF equation for $\Pi_{x_k}$ can be done by weakly differentiating $\Pi_{x_k}$ with respect to $t$ and employing the sifting property. However, by means of a mollifier argument, we give the formal derivation in the following theorem.
\begin{theorem} \label{thm}
$\Pi_{x_k}(X_k,t;\w)$ almost surely obeys, in the sense of distributions, the linear conservation law
    \begin{align} \label{ro-clpi}
        \frac{\d\Pi_{x_k}}{\d t} + \frac{\d}{\d X_k}\Big[v_k(X_k,\x_{-k}(t),t;\w)\,\Pi_{x_k}\Big] = 0, \qquad \Pi_{x_k}(X_k,0) = \delta\left(\x_k^0(\w)-X_k\right).
    \end{align}
\end{theorem}
\begin{proof}
Define $\Pi_\e(X_k,t)$, a regularized version of $\Pi_{x_k}(X_k,t)$ in~\cref{rawpdf2}, as 
\begin{align}\label{eq:Pi}
	\Pi_\e(X_k,t) \triangleq (\eta_\e \star \Pi_{x_k})(X_k,t) \triangleq \int_{\R} \eta_\e(X_k-Y)\Pi_{x_k}(Y,t) 	\,dY = \eta_\e(X_k-x_k(t)),
\end{align}
where the last equality holds by the definition of $\Pi_{x_k}(Y,t)$ and the sifting property of the Dirac distribution.
The standard positive mollifier $\eta_\e \in \mathscr{C}_c^{\infty}\left(\R\right)$ satisfies the conditions of symmetry, $\eta_\e(X_k-x_k(t))= \eta_\e(x_k(t)-X_k)$, and scaling
\begin{align}
	\eta_\e(Y) \triangleq \frac{\e^{-1}}{\int\eta \,dY}\,\eta\left(\frac{Y}{\e}\right), \quad 	\text{where } \eta(Y) 
	\triangleq 
	\begin{cases} 
		\exp\left(\frac{1}{|Y|^2-1}\right) &\text{ if } |Y| < 1\\ 
		0 &\text{ if } |Y| \ge 1.
	\end{cases}
\end{align}
Following standard arguments from~\cite{evans}, one can show that $\Pi_\e$ is a smooth approximation of $\Pi_{x_k}$. 
Let $\phi(X_k,t) \in \mathscr{C}_c^1\left(\R\times [0,\infty)\right)$. It follows from~\cref{eq:Pi} that
\begin{align}\label{I1}
	\mathscr{I} \triangleq \int_0^{\infty}\int_\R \Pi_\e(X_k,t)\frac{\partial \phi}{\partial t}
	(X_k,t) \,dX_k dt =  \int_0^{\infty}\int_\R \eta_\e(X_k-x_k(t))\frac{\partial \phi}{\partial t}
	(X_k,t) \,dX_k dt.
\end{align}
Integrating by parts in $t$ and applying the sifting property gives 
\begin{align*}
	\mathscr{I}  = \;& \int_0^{\infty}\int_\R \dot{\eta}_\e(X_k-x_k(t))
	v_k(x_k(t),\x_{-k}(t),t;\w)\phi(X_k,t) 
	\,dX_k dt \\ 
 - \;& \int_\R  \eta_\e\left(X_k-x_k^0\right) \phi(X_k,0) \,dX_k,\\
	= \;& \int_0^{\infty}\int_\R \int_\R \dot{\eta}_\e(X_k-Y)v_k(Y,\x_{-k}(t),t;\w)
	\Pi_{x_k}(Y,t)\phi(X_k,t) \,dY dX_k dt \\ - \;& \int_\R  \Pi_\e(X_k,0)\phi(X_k,0) dX_k,
\end{align*}
where $\dot{\eta}_\e( \cdot)$ is the derivative of $\eta_\e(\cdot)$. According to the Gauss-Ostrogradsky theorem in $X_k$,
\begin{align}\label{I2}
	\mathscr{I}  =  -\int_0^{\infty}\int_\R (\eta_\e \star v_k\Pi_{x_k})(X_k,t)\frac{\d \phi}{\d X_k}(X_k,t) \,dX_k dt - \int_\R  \Pi_\e(X_k,0)\phi(X_k,0) \,dX_k.
\end{align}
It follows from~\cref{I2} and~\cref{I1} that for any $\phi \in \mathscr{C}_c^1\left(\R\times [0,\infty)\right)$, 
\begin{align*}
	\int_0^{\infty}\int_\R \Pi_\e\frac{\partial \phi}{\partial t} \,dX_k dt 
	+ \int_0^{\infty}\int_\R (\eta_\e \star v_k\Pi_{x_k}) \frac{\d \phi}{\d X_k} \,dX_k dt 
	+ \int_\R  \Pi_\e(X_k,0)\phi(X_k,0) \,dX_k = 0.
\end{align*}
By standard arguments, taking the limit $\e \rightarrow 0$ gives
\begin{align*} 
	\int_0^{\infty}\int_\R \Pi_{x_k}\frac{\partial \phi}{\partial t} \,dX_k dt 
	+ \int_0^{\infty}\int_\R (v_k\Pi_{x_k}) \frac{\d \phi}{\d X_k}\,dX_k dt 
	+ \int_\R  \Pi_{x_k}(X_k,0)\phi(X_k,0) \,dX_k dt = 0;
\end{align*}
 hence, $\Pi_{x_k}$ is the distributional solution to \cref{ro-clpi}, which completes the proof. 
\end{proof}

Taking the ensemble mean of \cref{ro-clpi} over the space of $x_k(t)$, and applying the sifting property gives 
\begin{align} \label{eqn:RoPDF-derivation}
	\frac{\d f_{x_k}}{\d t}
	+ \frac{\d}{\d X_k}\int_{\R^{N-1}}\int_{\R^{N_p}}v_k(X_k,\X_{-k},t;\boldsymbol\Xi)
	f_{\x,\boldsymbol\xi}(\X,\boldsymbol\Xi;t)
	\,d\boldsymbol\Xi \,d\X_{-k} =0,
\end{align} 
where $f_{\x,\boldsymbol\xi}(\X,\boldsymbol\Xi;t)$ denotes the joint PDF of system states $\x(t)$ and random coefficients $\boldsymbol\xi(t)$. RoPDF equation \cref{eqn:RoPDF-derivation} is exact, but unclosed, since it depends on the generally unknown $f_{\x,\boldsymbol\xi}$ and not on $f_{x_k}$ alone.  However, factoring $f_{\x,\boldsymbol\xi}$ into the product of the marginal PDF $f_{x_k}$ and conditional PDF $f_{\mathbf{x}_{-k},\boldsymbol{\xi}|x_k}$, \cref{eqn:RoPDF-derivation} can be expressed in terms of the regression function $\mathcal{R}$:
\begin{align} \label{RoPDF3}
	\frac{\d f_{x_k}}{\d t}
	+ \frac{\d}{\d X_k}\big(\mathcal{R}(X_k,t) f_{x_k}\big) =0, \qquad f_{x_k}(X_k;0) = \int_{\R^{N-1}} f_{\x^0}(\X)\, d\X_{-k},
\end{align} 
together with vanishing boundary conditions, where 
\begin{equation} \label{eqn:rf}
    \mathcal{R}(X_k, t) \triangleq 
    \big\langle v_k(X_k, \mathbf{x}_{-k}(t), t; \w) \,\big| \,x_k(t) = X_k \big\rangle.
\end{equation} 
is to be estimated from data. 

In its current form, \cref{RoPDF3} is fully data-driven, completely relying on accurate estimation of \cref{eqn:rf}. However, many applications produce a regression function that is partially, if not fully, separable in the QoI $x_k(t)$. By this, we mean the $k$-th velocity component can be decomposed into $v_k(\x,t;\w) = \sum_{i\in I} g_i(x_k,t)h_i(\x,t;\w)$ for some finite collection of known real-valued functions $\{g_i, h_i\}_{i\in I}$. Then, each $g_i(X_k,t)$ may be pulled outside the conditional expectation \cref{eqn:rf} and need not be estimated, giving the following physics-informed representation of \cref{RoPDF3}:
\begin{align} \label{RoPDFsep}
	\frac{\d f_{x_k}}{\d t}
	+ \frac{\d}{\d X_k}\left[\left(\sum_{i\in I}g_i(X_k,t)\rmc_i(X_k,t)\right) f_{x_k}\right] =0,
\end{align} 
with new regression functions $\rmc_i(X_k,t)\triangleq \langle h_i(X_k,\x_{-k}(t),t;\w) \,|\, x_k(t) = X_k \rangle$. If $h_i$ has no dependence on $x_k$, i.e., $h_i \equiv h_i(\x_{-k},t;\w)$, for all $i\in I$, then the regression function \cref{eqn:rf} is considered fully separable with respect to the QoI. In both settings, part of the advection coefficient is known in closed-form, reducing the amount of data needed for accurate RoPDF solutions, as was investigated for non-stiff, noiseless RODEs in \cite{BRENNAN2018281} and It\^o SDEs in \cite{MaltbaPES}. 

As is typical in UQ, uncertainty in \cref{ivp2} has been fully prescribed. Hence, corresponding state data to be used for regression is synthetically (and independently) generated by numerically integrating \cref{ivp2}. However, since we are concerned with stiff, multiscale RODEs, costly implicit schemes are required for this data generation, leading to limited data availability. In other words, regression is performed in the small-sample regime, which calls for more expensive, robust algorithms. This issue is exacerbated for systems exhibiting strong nonlinearities, where nonparametric methods must be employed as in \Cref{sec:power}. Moreover, regression functions associated with multiscale RODEs may vary considerably on short timescales and induce a large RoPDF equation Courant number, requiring regression estimates at an unusually large number of discrete times. However, we drastically reduce our computational overhead by considering only simple, non-robust regression at sparse observation times. Naturally, this simplification amounts to misspecifying the governing RoPDF equation and introduces nonnegligible errors, which we control by sparsely assimilating in low-fidelity RoPDF estimates. The result is a method whose computational demand is almost entirely associated with the overhead of synthetic state-data generation via (relatively few) MC realizations of \cref{ivp2}, while preserving the qualitative behavior of the original equation.


\section{Data Assimilation}
\label{sec:assim}

To reduce error introduced by sparse observation times associated with stiff, multiscale systems, we frame the RoPDF method as a data assimilation problem. Arguably, the two most commonly employed assimilation procedures for hyperbolic PDEs are 4D-Var \cite{Dimet_2010} and the ensemble Kalman filter (EnKF). However, neither are particularly well-suited for the RoPDF method. The former relies on a computationally demanding global optimization procedure and the latter suffers from the curse of dimensionality, making it ill-suited for discretized PDEs. Moreover, the EnKF performs poorly when PDE observations are noisy with large and/or highly non-Gaussian errors \cite{jazwinski1970, Lei_2015}. However, one workaround, and the first assimilation procedure under consideration is nudging (NR), where the PDE correction term is designed to converge quickly to zero in one forward simulation. Moreover, the nudging appellation motivates our second, global approach, where we make use of DNNs. Although it is not dynamic assimilation, DNNs can be viewed as instantaneous relaxation, which can address some of NR's shortcomings, such as (temporal) sparsity of available data.

To formulate the assimilation problems, suppose we have generated $N_\text{MC}^\text{tr}$ MC realizations of \cref{ivp2} (i.e., training data) so that $\rmc$ may be approximated by a smooth estimator $\hat\rmc$. Letting $\E(X_k, t) \triangleq \rmc(X_k, t) - \hat\rmc(X_k, t)$ denote the corresponding residual arising from the regression, the RoPDF equation \cref{RoPDF3} can be identically written
\begin{align} \label{RoPDF4}
    \frac{\d f_{x_k}}{\d t} + \frac{\d}{\d X_k} \big(\hat\rmc\, f_{x_k} \big) = \big\langle\M\big\rangle,
\end{align}
where $\langle\M\rangle \equiv -\d_{X_k}(\E f_{x_k})$, referred to as the model defect/discrepancy, is unknown {\emph{a priori}}. Note that we have used the fully data-driven RoPDF representation \cref{RoPDF3} simply for notational brevity. In practice, the advection coefficient takes the form of the physics-informed version \cref{RoPDFsep}, albeit with $\hat\rmc_i$ in place of $\rmc_i$. By means of NR and DNNs, we learn the model defect by assimilating in RoPDF observations $H(f_{x_k})$, where $H(\cdot)$ represents a given RoPDF observation map. 


\subsection{Nudging}
\label{sec:nr} 

To reduce RoPDF discrepancy, NR assumes the model defect can be described by a simple correction. The resulting PDE for the observer/estimator $\hat f_{x_k}^\text{NR}$ takes the form
\begin{align} \label{fnud}
	\frac{\d \hat f_{x_k}^\text{NR}}{\d t} + \frac{\d}{\d X_k}\big(\hat \rmc\,\hat f_{x_k}^\text{NR}\big)
	= \lambda\big(H(f_{x_k})-\hat f_{x_k}^\text{NR}\big)
	,\qquad \hat f_{x_k}^\text{NR}(X_k;0) = f_{x_k}(X_k;0),
\end{align}
with boundary conditions identical to those in \cref{RoPDF3}. Here, the observation map $H(\cdot)$ accounts for data availability and sparsity, i.e., when observations of $f_{x_k}$ are possibly noisy and known only on a subset of spatiotemporal locations of the domain. Taking  $H(\cdot)$ to be the identity map implies that complete, exact observations are available. The NR coefficient $\lambda>0$ acts as a finite learning rate that dynamically relaxes the observer towards the observations, controlling the convergence of $\hat f_{x_k}^\text{NR}$ to $f_{x_k}$. The choice of $\lambda$ is largely empirical and typically requires some level of manual tuning. This is in contrast to the EnKF, which takes $\lambda$ to be the Kalman gain matrix, requiring Gaussian error distributions. In practice, the observations of $f_{x_k}$ are typically noisy to some degree. If they are indeed assumed to be perfectly random, it can be shown that the correction in \cref{fnud} is equivalent to scaled white noise in a stochastic PDE, as discussed in \cite{Boulanger_2015} and the references therein. This equivalence was originally given by Jarwisnky \cite{jazwinski1970} between nudged ODEs and SDEs. At a high level, this explains why the Kalman gain matrix is optimal when error distributions are Gaussian, assuming the underlying dynamics are linear. However, practical NR ignores this introduced uncertainty to a certain level, allowing $\lambda$ to be manually tuned to fit the data or used for forecasting when the criteria for EnKF are not met. Moreover, when observations are sparse, $\lambda$ can be constructed to vary in space and/or time, classically comprised of weight functions. Another option is to interpolate observations to the full computational domain. General strategies for constructing $\lambda$ are reviewed in \cite{Lakshmivarahan_2013}.

\begin{remark} \label{rm:def}
The $\langle\cdot\rangle$ notation used in the correction of \cref{RoPDF4} refers to the defect being a homogenized quantity. This is to maintain notational consistency with the existing literature on nudged PDEs \cite{Boulanger_2015}, where NR is reformulated on the ``microscopic level'' by using the PDE's kinetic description. In the setting of PDF/CDF equations, this was studied in \cite{Boso_2020} for nonlinear hyperbolic PDEs with random initial data. To the best of our knowledge, we are the first to consider it for RoPDF equations, where the kinetic formulation amounts to nudging the raw RoPDF equation \cref{ro-clpi} with (possibly noisy, sparse) observations $H(\Pi_{x_k})$, for which the strong convergence results of \cite{Boulanger_2015} apply. By the triangle inequality, convergence of the microscopic observer $\hat\Pi_{x_k}$ to $\Pi_{x_k}$ implies $L_1$ convergence of a corresponding macroscopic nudged observer, which we thoroughly discuss in \Cref{app:nr}.
\end{remark}


\subsection{Deep Neural Networks}
\label{sec:dnn}

A potential drawback of NR is that qualitative properties of the true RoPDF cannot be guaranteed for the observer, particularly when the observations are noisy and/or sparse. Although these issues are not present for the applications in \Cref{sec:experiments}, this generally may not be the case. For \cref{fnud}, its solution $\hat f_{x_k}^\text{NR}$ is not guaranteed to have the PDF properties of nonnegativity and unit mass. One alternative is the use of DNNs for direct RoPDF inference from observations, where regularity terms may be added to the DNN loss function to enforce PDF properties and appropriate boundary conditions if necessary.

To reformulate the NR problem \cref{fnud} as an instantaneous one via a DNN, we introduce an optimization problem over the (sparse) spatiotemporal observation points of the domain $(X_k,t)$.  We denote the vector of $N_\text{obs}$-many discrete observation locations by $\tilde\X_k^\nu =\left[\X_k^\top, \mathbf T_\nu^\top\right]^\top$, where $\X_k$ and $\mathbf T_\nu$ represent the spatial and temporal components, respectively. The subscript $\nu\in\mathbb{N}$ denotes the level of temporal sparsity of the data, which is formally defined in \Cref{sec:num}. The optimization problem is then defined by minimizing the discrepancy between the observer $\hat f_{x_k}^\text{DNN}$ and observations via the following loss function:
\begin{align} \label{nnloss}
	\mathscr L \triangleq \left|\left|\hat f_{x_k}^\text{DNN}(\X_k;\mathbf T_\nu) 
	- H(f_{x_k}(\X_k;\mathbf T_\nu))\right|\right|,
\end{align}
where $||\cdot||$ is an appropriate norm over $\R^{N_\text{obs}}$. Directly approximating $\hat f_{x_k}^\text{DNN}$ in \cref{nnloss} as a DNN is a possibility, but given that the result would be solely data-driven, utilizing the partially known dynamics (i.e., the separable advection term) of \cref{RoPDFsep} and \cref{RoPDF4} gives better results due to increased statistical power. To incorporate such dynamics and render the loss \cref{nnloss} physics-informed, we utilize the RoPDF equation's linearity. 

The solution $f_{x_k}$ to \cref{RoPDF4} can be decomposed into the sum of its homogeneous and particular (defect) solutions $f_{x_k}^\text{h}$ and $f_{x_k}^\text{d}$, respectively, such that $f_{x_k} = f_{x_k}^\text{h} + f_{x_k}^\text{d}$. Naturally, $f_{x_k}^\text{h}$ is the solution to the homogeneous equation \cref{RoPDF4} (i.e., when $\langle\M\rangle\equiv0$), while $f_{x_k}^\text{d}$ accounts for the defect's contribution. This fact can be established by applying the method of characteristics to \cref{RoPDF4} via the terminal value problem
\begin{align} \label{tvp}
	\frac{d \chi(s)}{ds} = \hat\rmc(\chi(s),s), \qquad \chi(t)=X_k,
\end{align} 
and its associated flow $\chi(s)\equiv\Phi(s;X_k,t)$ for $0\le s <t$. By restricting $f_{x_k}$ along the characteristic curves, the RoPDF equation can be solved via integrating factor, resulting in 
\begin{align} \label{fparts}
	f_{x_k}^\text{h}(X_k;t) &= \J(0;X_k,t) f_{x_k}(\Phi(0;X_k,t);0), \notag\\
	f_{x_k}^\text{d}(X_k;t) &= \int_0^t \langle\M\rangle
	(\chi(r),r)\J^{-1}(r;X_k,t)\,dr,
\end{align}
where
$
	\J(s;X_k,t)\triangleq \exp\left(-\int_s^t \d_\chi \hat\rmc(\chi(r),r)
	\right)\, dr.
$

The homogeneous solution $f_{x_k}^\text{h}$ in \cref{fparts} is directly computed via numerical integration. This is done by  separating the advection coefficient into its known and unknown terms as in \cref{RoPDFsep}, approximating $\rmc_i$ with smooth estimators $\hat\rmc_i$ on the spatial mesh $\X_k$ for each (sparse) observation time in $\mathbf T_\nu$. When $\nu>1$, the learned $\hat\rmc_i$ may be interpolated to the dense temporal grid required by the homogeneous PDE discretization to improve performance. Having $f_{x_k}^\text{h}$ at our disposal, we construct an instantaneous observer $\hat f_{x_k}^\text{DNN} \triangleq f_{x_k}^\text{h} + \hat f_{x_k}^\text{d}$ for $f_{x_k}$ by estimating $f_{x_k}^\text{d}$ with a fully connected feedforward DNN $\hat f_{x_k}^\text{d}$ containing $N_\text{lay}$ layers:
\begin{align} \label{mlp}
	\hat f_{x_k}^\text{d}(X_k;t) \triangleq \A_{N_\text{lay}} \circ
	\phi \circ \A_{N_\text{lay}-1}  \circ \cdots \circ \phi \circ \A_1\tilde\X_k^\nu,
\end{align}  
where $\phi$ is a nonlinear activation function applied recursively to each of the $N_\text{lay}-1$ hidden layers. Note that since $\phi$ is typically bounded from above and/or below, it is not applied to the output layer $\A_{N_\text{lay}}$ since our intended purpose is regression. Since $f_{x_k}^\text{d}$ accounts for the defect's contributions to the RoPDF $f_{x_k}$, its qualitative behavior can be complex, i.e., nonperiodic with steep gradients. DNNs are an expressive hypothesis class, and are known to learn complex function behavior, which is the reasoning behind this choice of observer. Moreover, partial separability of the advection coefficient ensures that $f_{x_k}^\text{h}$ and therefore $\hat f_{x_k}^\text{DNN}$ is physics-informed, resulting in increased predictive power. 

Since there are no existing theoretical results for the convergence of the DNN observer, the choice of norm is not as restrictive as in NR. We take the standard mean squared error (MSE) since it gives a smooth, convex loss, significantly reducing computational costs compared to the nondifferentiable $L_1$ loss, but it is not without caveats. Employing the MSE loss for DNN regression may result in poor training convergence if the underlying error distributions in the observations $H(f_{x_k})$ are not close to being independent, identical, and Gaussian. The MSE can be replaced with a more general loss function to address errors that strongly violate these properties. For example, to account for non-constant variance, one can use the generalized least squares (GLS) loss as in \cite[Eq. 6]{Lagergren2_2020}, where DNNs (specifically physics-informed neural networks (PINNs) \cite{RAISSI2019686}) were trained with the GLS loss to improve training convergence. 
Moreover, if the resulting observer does not have the desired properties of a PDF, regularity terms may be added to the loss. For example, to enforce nonnegativity, the penalty
\begin{equation} \label{eqn:dnn-nonnegative-penalty}
    \frac{1}{N_\text{obs}} \sum_{\{i:\hat f_{x_k}^\text{DNN}<0\}} \big(\hat f_{x_k}^\text{DNN}(\tilde{X}_{k,i}^\nu)\big)^2
\end{equation} may be included in the loss. Similar penalties may also be added to enforce unit mass and boundary conditions. However, while effective, using generalized loss functions or regularity terms can increase training costs, which occurred for the \Cref{sec:experiments} applications. Therefore, the experiments presented use the standard MSE loss but with problem-specific transformations applied (before training) to both the predictor (observation location) and response ($H(f_{x_k})$) data to account for vastly differing scales and a variety of complex error distributions.

\begin{remark}
	Instead of the DNN formulation above, a PINN may be employed, which would simultaneously learn $\rmc_i$ and the solution to \cref{RoPDFsep} via stacked DNNs. Due to the behavior of $\rmc_i$ in \Cref{sec:power}, a PINN formulation decreased predictive accuracy and increased training costs. This is likely due to the highly nonconvex loss landscapes associated with PINN approaches to advection-dominant PDEs, as discussed in \cite{NEURIPS2021_df438e52}.
\end{remark}


\section{Numerics}
\label{sec:num}

Via the order 1.5 implicit strong Taylor scheme from \cite[Ch. 10.2]{Han2017}, state data is synthetically generated by MC simulations of \cref{ivp2} and collected on a set of uniform times $\mathbf{T}_1= \{t_m\}_{m=0}^{M}\triangleq\{m\Delta t\}_{m=0}^{M}$, where $t_{M}=T_f$. For each $t_m\in\mathbf{T}_1$, a large number of $N_\text{MC}$ MC samples of the QoI $x_k(t_m)$ are post-processed with robust, adaptive-like KDE \cite{Botev_2010} to form a MC marginal PDF solution $f_\text{MC}(X_k; t_m)$, which is treated as a yardstick for testing the RoPDF method. Naturally, $N_\text{MC}$ is problem dependent and is determined by means of a convergence study for each experiment. 

We introduce the notion of sparse data via a sparsity factor $\nu\in\{1,\dots,M\}$ and its associated observation times $\mathbf{T}_\nu = \{t_{m_l}\}_{l=0}^{M_\nu} \triangleq \{\nu l \Delta t\}_{l=0}^{M_\nu}$, where $M_\nu\le M$. Hence, $\nu=1$ implies complete observations, $\nu=2$ implies every other observation is available, and so on. Independent of the trials for $f_\text{MC}$, we perform $N_\text{MC}^\text{tr}\ll N_\text{MC}$ MC simulations of \cref{ivp2} to collect state ($\x$) and, if required, noise ($\boldsymbol\xi$) data for training. For a given $\nu$, at each time $t_{m_l}\in\mathbf T_\nu$, these samples are used to learn $\hat\rmc_i(X_k,t_{m_l})$ via linear (ordinary least squares (OLS)) regression and Gaussian local linear regression (GLLR) \cite{Hastie_2009} for the linear and power systems applications in \Cref{sec:experiments}, respectively. Additionally, the $x_k(t_{m_l})$ samples are post-processed with KDE \cite{Botev_2010} to compute a low-fidelity RoPDF observation $H(f_{x_k}(X_k;t_{m_l}))$ for each observation time. Since $N_\text{MC}^\text{tr}\ll N_\text{MC}$, these RoPDF observations are inherently noisy, which is exacerbated in the temporal domain for $\nu>1$. We henceforth denote these observations by $f_\text{MC}^{\text{tr},\nu}$, to identify their dependence on the training sample size, KDE, and the sparsity level. Much of our analysis will focus on how training size and sparsity level affect observer accuracy.

In all experiments that follow, the homogeneous RoPDF equations are solved on the set of dense times $\mathbf T_1$ via a Lax-Wendroff finite volume discretization with a monotonized central limiter. When observation times are sparse, this is done by interpolating the learned regression functions $\hat\rmc_i$ to the dense spatiotemporal grid $(\X_k,\mathbf T_1)$ via 2D modified Akima interpolation \cite{makima}. Although the phase space is unbounded in our formulation, the computational spatial domain is taken to be a sufficiently large (bounded) interval so that vanishing boundary conditions at $\pm\infty$ may be approximated with homogeneous Dirichlet conditions. 


\subsection{Assimilation Training}
\label{sec:train}

The nudged equation \cref{fnud} is solved by successively considering the homogeneous advection and source equations via Strang operator splitting, where the source equation is integrated with a Crank-Nicolson discretization. Similar to $\hat\rmc_i$, for $\nu>1$, observations $f_\text{MC}^{\text{tr},\nu}$ are also interpolated to the dense mesh before numerically integrating, which avoids the laborious tuning of NR weight functions. We take the relaxation rate $\lambda\equiv\lambda_\nu(t)$ to be piecewise constant over observation intervals $[t_{m_l},t_{m_{l+1}})$, which is tuned in an online fashion to reduce predictive error at the subsequent observation times. In particular, for a given interval with $t\in[t_{m_l},t_{m_{l+1}})$ and $\nu>1$, we consider two possible values: $\lambda_\nu(t)\equiv0$ and $\lambda_\nu(t)\equiv\nu$. Supposing $\lambda_\nu$ and $\hat f_{x_k}^\text{NR}$ have been computed for $t<t_{m_l}$, we solve \cref{fnud} up to the following observation time $t_{m_{l+1}}$ for both possible values of $\lambda_\nu$. Whichever produces the lowest ($L_1$) prediction error between the observer $\hat f_{x_k}^\text{NR}$ and the observation $f_\text{MC}^{\text{tr},\nu}$ at time $t_{m_{l+1}}$ is taken as $\lambda_\nu(t)$ on $t\in[t_{m_l},t_{m_{l+1}})$. This approach to tuning ensures that observations are assimilated into RoPDF dynamics only when necessary, significantly improving the purely scalar NR approach in \cite{maltbaphd}. 

For both applications that follow, in the DNN formulation, we represent the defect solution $\hat f_{x_k}^\text{d}$ as a fully connected DNN with a ReLU  activation function. For each combination of $\nu$ and $N_\text{MC}^\text{tr}$, we train a DNN via the standard MSE loss \cref{nnloss} using a $30\%$ holdout set for model validation. The MSE is minimized via the L-BFGS optimizer in PyTorch v1.13.0 \cite{pytorch} with a maximum of $5\times10^3$ iterations. To prevent overfitting, we implement an early-stopping criterion by imposing a $10^{-8}$ gradient tolerance, which allowed training to terminate in at most $10^3$ iterations. We consider $3$ to $10$ equally sized hidden layers, ultimately choosing the network depth that minimizes validation MSE. For \Cref{sec:lin}, the network width is fixed at $20$ neurons, which is subsequently increased to $32$ neurons for \Cref{sec:power}. In both applications, $N_\text{MC}^\text{tr}$ has little effect on optimal network depth so long as it is not overwhelmingly small relative to dynamic complexity, e.g., greater than $250$ and $10^3$ for the linear and powers systems, respectively. $\nu\ge1$, on the other hand, is much more influential on optimal depth. This is not surprising considering that sparse $\hat\rmc_i$ may introduce large RoPDF errors for systems that are multiscale and/or rapidly oscillating, resulting in defects of greater complexity. After the training period, for each $\nu$, an $\hat f_{x_k}^\text{d}$ prediction is computed on the set of complete times $\mathbf T_1$ required by the homogeneous equation's discretization. It is added to the homogeneous solution $f_{x_k}^\text{h}$ to obtain the instantaneous observer $\hat f_{x_k}^\text{DNN}$. 


\subsection{Computational Complexity}
\label{sec:complexity}

Unlike DNNs, it is straightforward to compute the homogeneous and NR equations $\bo(\cdot)$ complexities. For simplicity, suppose the 1D spatial domain is discretized with a fixed uniform mesh $\X_k$ containing $N_{x_k}$ cells. Likewise, the dense temporal grid $\mathbf T_1$ contains $M+1$ nodes. Assume $N_\text{MC}^\text{tr}$ MC realizations of \cref{ivp2} have been computed and stored at the $M_\nu \approx \lceil M/\nu \rceil$-many times $\mathbf T_\nu$ ($\nu\ge1$).
 
 The homogeneous equation coincides with NR \cref{fnud} when $\lambda\equiv0$ for all $t\in[0,T_f]$. Given an advection coefficient, a Lax-Wendroff time step requires $\bo(N_{x_k})$ operations, and therefore a total of $\bo(MN_{x_k})$ operations over $[0,T_f]$. Given the Courant-Friedrichs-Lewy (CFL) condition to ensure numerical stability, this can be expressed as $\bo(N_{x_k}^2)$. However, we must account for the cost of learning $\hat\rmc$. Consider the more expensive nonparametric GLLR regression.
 For a given $t_{m_l}\in\mathbf T_\nu$ and bandwidth parameter, GLLR fitting and evaluation has $\bo(N_\text{MC}^\text{tr} N_{x_k})$ complexity \cite[Ch. 6.9]{Hastie_2009}. A typical CV procedure for bandwidth selection increases this cost from linear to quadratic in $N_\text{MC}^\text{tr}$. However, we avoid CV by transforming the data so that the simple plug-in estimator \cref{eq:bw} is sufficiently accurate, keeping GLLR $\bo(N_\text{MC}^\text{tr} N_{x_k})$, and therefore $\bo(M_\nu N_\text{MC}^\text{tr} N_{x_k})$ over all observation times. The cost of employing any piecewise cubic Hermite interpolating polynomial is at most $\bo(M N_{x_k})$, giving a total complexity of
 \begin{equation} \label{eq:bo}
    \bo\left(M_\nu N_\text{MC}^\text{tr}N_{x_k} + M N_{x_k}+ M N_{x_k}\right) = \bo\left(\left(N_\text{MC}^\text{tr}/\nu + 2\right) N_{x_k}^2\right)
 \end{equation}
for the homogeneous RoPDF equation. 

In addition to \cref{eq:bo}, to solve the nudged equation \cref{fnud}, we must account for Strang splitting and the KDE/interpolation procedure for $f_\text{MC}^{\text{tr},\nu}$. However, the cost of computing $f_\text{MC}^{\text{tr},\nu}$ at times $\mathbf T_\nu$ via KDE and interpolating to the dense mesh is identical to the advection coefficient procedure. Moreover, a single time step of the source equation takes $\bo(N_{x_k})$ operations, and taking $M$ to be even only requires an additional $M$ time steps for Strang splitting (compared to the standard $2M$ additional steps). Thus, the computational complexity of solving the nudged equation \cref{fnud} is twice that of the homogeneous equation.

Circling back to the overhead of generating $N_\text{MC}^\text{tr}$ realizations of the $N$-dimensional RODE \cref{ivp2}, for a given path realization, the majority of costs corresponds to the nonlinear/implicit solve required at each time step. If we assume that the mean velocity field's Jacobian is known exactly, then each iteration of a Newton-type method can be done in $\bo(N^2)$ to $\bo(N^3)$ operations, depending on the system and matrix factorization. Recalling $\bo(M)=\bo(N_{x_k})$, this amounts to an overhead of $\bo(N_\text{MC}^\text{tr} N_{x_k} N^3)$ operations for an arbitrary stiff, nonlinear system. However, we have not accounted for multiple iterations during nonlinear solves nor Jacobian approximations. Hence, true overhead costs may be considerably larger, especially for very stiff and/or high-dimensional RODEs. Regardless, even for moderate $N$, sampling \cref{ivp2} dominates \cref{eq:bo}, and therefore the cost of the nudged RoPDF method.


\section{Experiments}
\label{sec:experiments}

We test the proposed RoPDF approaches on two applications, both with random initial data and driven by OU noise. The first in \Cref{sec:lin} is a stiff, 2D (i.e., $N=2$) linear system, included as a proof-of-concept. It highlights the need for data assimilation in RoPDF equations associated with sparsely observed stiff RODEs, even when the underlying dynamics are relatively simple. Our second application in \Cref{sec:power} is to power systems, where the RoPDF method is used for UQ of transmission/line failures in an electrical power grid. The governing model is a highly stiff, 47D nonlinear system. 


\subsection{Stiff Linear System}
\label{sec:lin}
We first consider the following linear RODE system:
\begin{align} \label{eq:linrode}
\dot x_1 &= -2x_1 + x_2 +2\sin(t), \qquad\qquad\qquad\qquad\qquad\qquad\,\, x_1(0) \sim \mathcal{N}\left(2, 0.15^2 \right), \notag\\
\dot x_2 &= (\alpha-1)x_1 - \alpha x_2 + \alpha(\cos(t)-\sin(t)) + \sigma\xi(t), \qquad  x_2(0) \sim\mathcal{N}\left(3, 0.15^2 \right),
\end{align}
to be solved up to $T_f=10$, where the Gaussian initial conditions and noise (i.e., $x_1(0)$, $x_2(0)$, $\xi(t)$) are all taken independent of one another. The driving colored noise is taken as an OU process defined as the solution to the It\^o SDE 
\begin{align} \label{eq:ou}
d\xi(t) = -\frac{\xi(t)}{\tau}dt + \sqrt{\frac{2}{\tau}}dW(t), \qquad \xi(0)\sim\mathcal{N}(0,1),
\end{align}
where $W(t)$ is a standard Wiener process independent of $\xi(0)$. Given this initial condition, $\xi(t)$ is an exponentially correlated stationary Gaussian process with correlation length $\tau>0$. Its solution is conditionally given by the scaled, time-transformed Wiener process 
\begin{align} \label{eq:ou2}
\xi(t) = \xi(0)e^{-t/\tau} + W\left(1-e^{-2t/\tau}\right).
\end{align}
Hence, paths of $\xi$ can be directly sampled from the laws of $\xi(0)$ and $W$, which is more accurate and efficient than numerically integrating \cref{eq:ou}. The intensity of $\xi$ is denoted by $\sigma>0$. Lastly, $\alpha>0$ serves as a stiffness parameter, making \cref{eq:linrode} stiff when $\alpha\gg1$. In the experiments that follow, we set $\alpha=999$, $\sigma=100$, $\tau=0.1$, and $k=1$ such that the QoI is $x_1(t)$.

\subsubsection{RoPDF Equation \& Numerics}
\label{sec:lin_ropdf}

\begin{figure}[t!] 
	\centering  
    \includegraphics[width=.3285\linewidth]{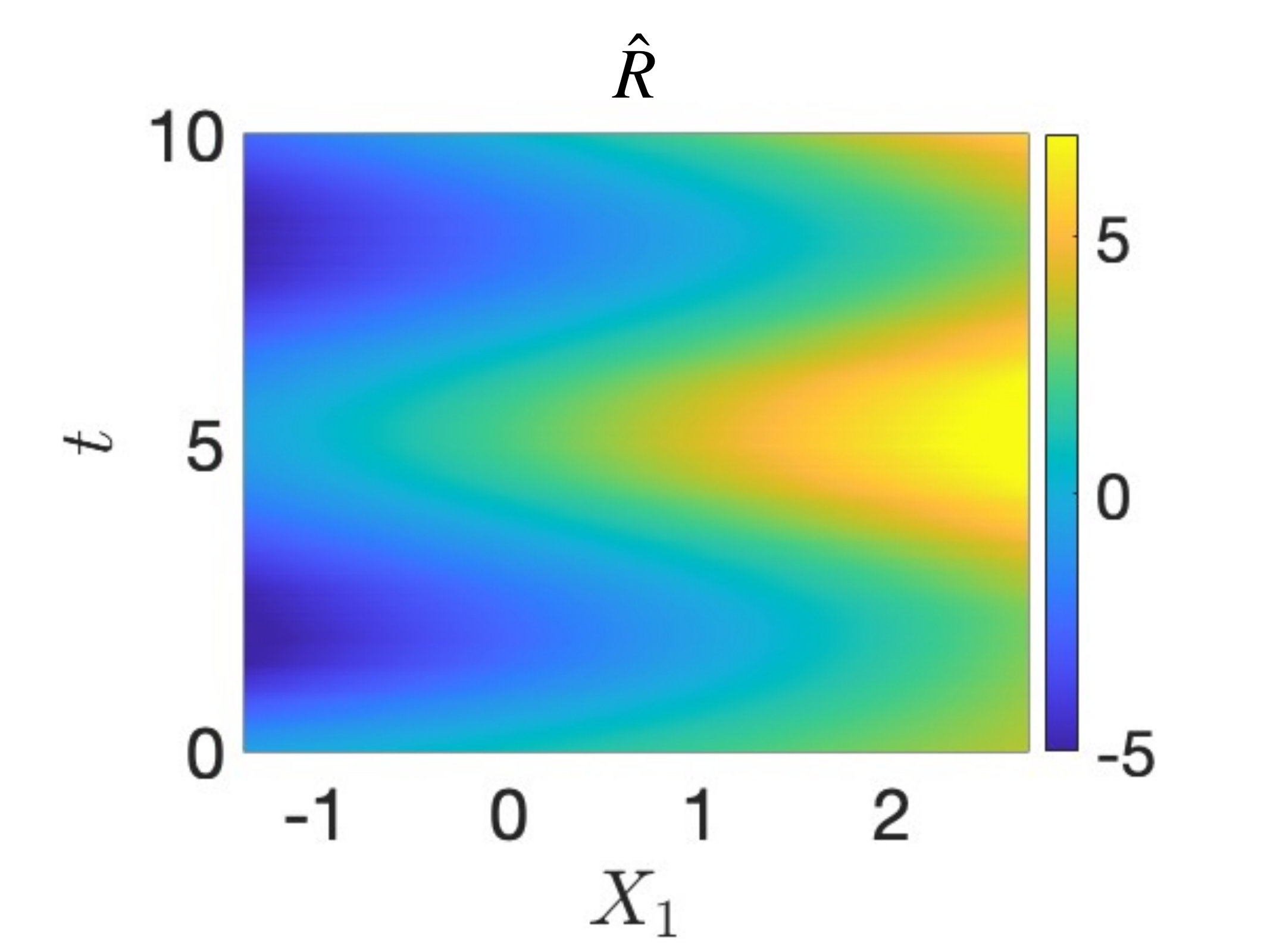}
    \includegraphics[width=.3285\linewidth]{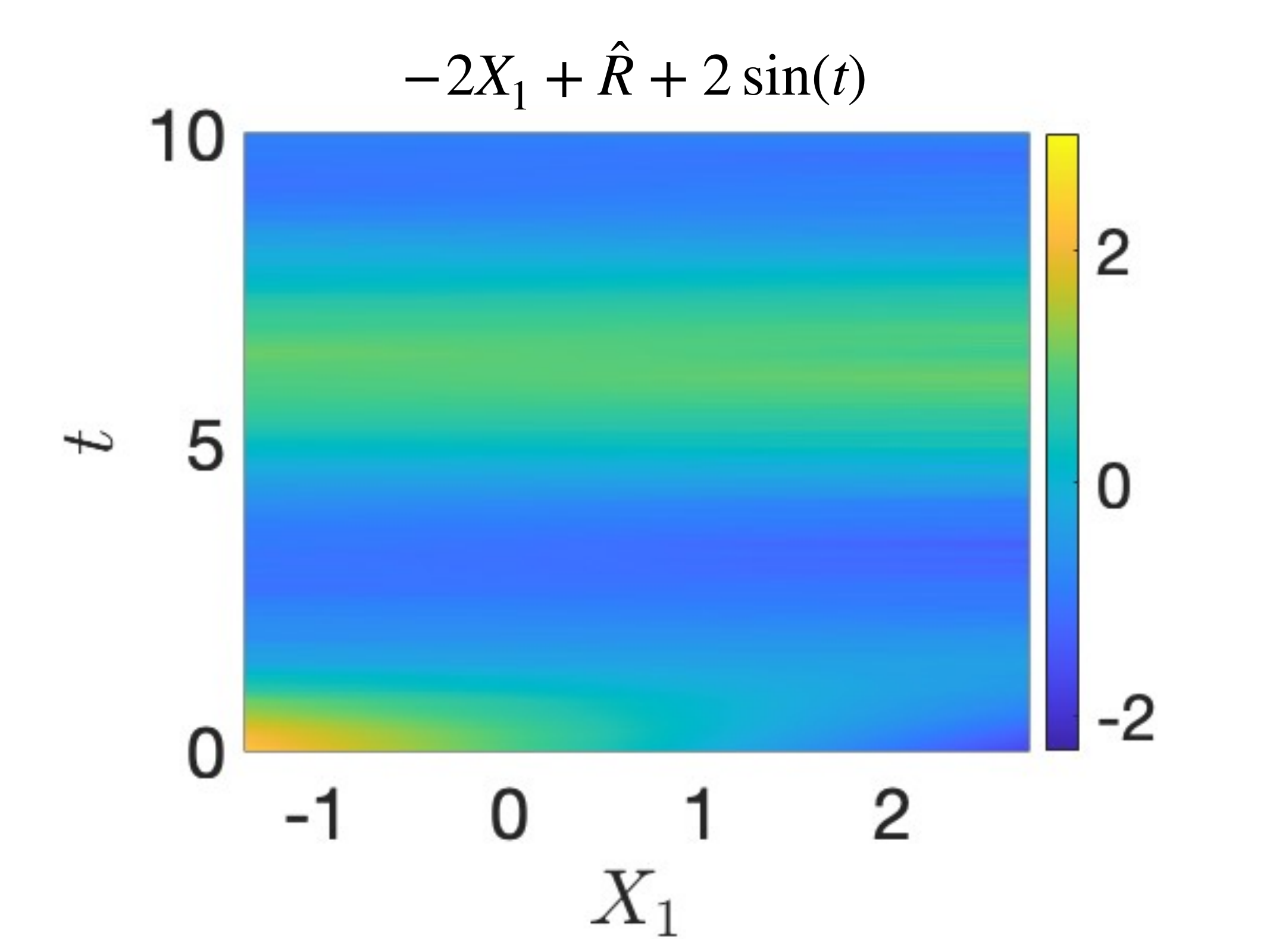}
    \includegraphics[width=.3285\linewidth]{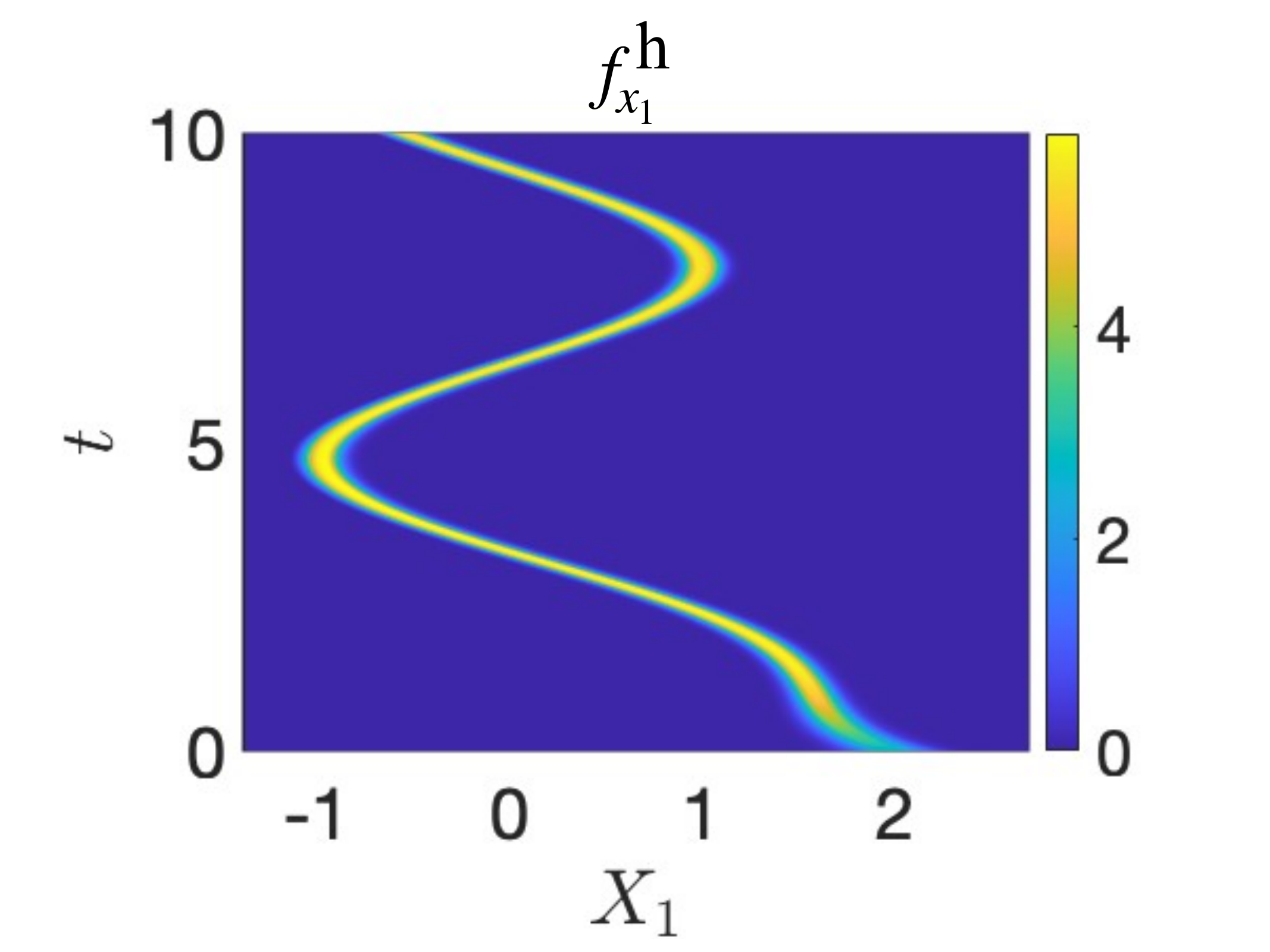}
\caption{(Left) Learned $\hat\rmc$ from $N_\text{MC}^\text{tr} = 5\times 10^2$ MC realizations of \cref{eq:linrode} at sparse times $\mathbf T_\nu$ with $\nu=2\times 10^2$. (Middle) The learned advection coefficient of \cref{eq:linpdf}. (Right) Evolution of the homogeneous solution $f_{x_1}^\text{h}$ to \cref{eq:linpdf}.}
\label{fig:lin_sol}
\end{figure}

The RoPDF equation \cref{RoPDFsep} takes the form
\begin{align} \label{eq:linpdf}
\frac{\d f_{x_1}}{\d t} + \frac{\d}{\d X_1}\Big[\Big(-2X_1 + \rmc(X_1,t) + 2\sin(t)\Big)\,f_{x_1}\Big] = 0,
\end{align}
where the initial condition $f_{x_1}(X_1;0)$ is the univariate Gaussian PDF of $x_1(0)$ and $\rmc(X_1,t) \triangleq \langle x_2(t)\,|\,x_1(t)=X_1\rangle$. The spatial mesh $\X_1$ is taken uniformly on $[-1.85, 3.15]$ with $5\times10^2$ cells such that $\Delta X_k = 10^{-2}$. This is a fairly dense mesh for the given dynamics; however, it demonstrates one manner in which sparse observations arise even when the dynamics are straightforward. For the given $\Delta X_k$ and estimated coefficient, the CFL condition requires $\Delta t \lesssim 1.4\times 10^{-3}$. Since this estimated bound is dependent upon the specific regression and interpolation methods, we take a slightly smaller time step $\Delta t = 10^{-3}$ for the dense grid $\mathbf T_1$. Given the simple structure of \cref{eq:linrode}, $\rmc$ is linear in $X_1$. Hence, the estimator $\hat\rmc$, which is displayed in \Cref{fig:lin_sol} along with its associated RoPDF evolution, is computed via OLS regression from $N_\text{MC}^\text{tr}$ MC realizations of $\x$ at each $t_{m_l}\in\mathbf T_\nu$.

To be consistent with the notation in our the DNN formulation, we denote the solution to \cref{eq:linpdf} (with $\hat\rmc$) by $f_{x_1}^\text{h}$, which signifies that no RoPDF observations were assimilated. We incorporate varying degrees of data sparsity by considering $\nu\in\{1, 2, 5\}\times10^2$, which corresponds to data availability at time increments of $0.1$, $0.2$, and $0.5$. Several training sample sizes were also considered in our experiments, but apart from our scalability results (\Cref{fig:lin_scale}), we limit our head-to-head comparisons to the $N_\text{MC}^\text{tr} = 5\times 10^2$ case, which is considerably fewer than the $N_\text{MC} = 1.5\times 10^4$ realizations needed to compute yardstick solution $f_\text{MC}$.

As mentioned in \Cref{sec:dnn}, the loss \cref{nnloss} in the DNN formulation is not actually minimized over the observation locations $\tilde\X_1^\nu$. Since the PDFs are near-Gaussian, for a given $\nu$, we compute the mean and standard deviation of $f_{x_1}^\text{h}(\X_1;t_{m_l})$ for each $t_{m_l}\in\mathbf T_\nu$. The spatial observation locations are then shifted and scaled by the corresponding mean and standard deviation for each time. Both $f_{x_1}^\text{h}$ and $f_\text{MC}^{\text{tr},\nu}$ are also scaled by these standard deviations. The resulting transformations result in PDFs that are nearly standard Gaussian for all observation times, albeit defined on varying/moving spatial grids. This method of standardization allows us to omit, for all $t_{m_l}\in\mathbf T_\nu$, any transformed spatial location with magnitude greater than four, i.e., where (transformed) $f_{x_1}^\text{h}$ and $f_\text{MC}^{\text{tr},\nu}$ are within machine epsilon. This improves DNN costs by reducing training input size and allows training to converge with shallower networks. After training, the $\hat f_{x_1}^\text{d}$ prediction on dense $\mathbf T_1$ is transformed back to original scale on $\X_1$.


\subsubsection{Error Analysis}
\label{sec:lin_err}

\begin{figure}[t!] 
	\centering  
    \includegraphics[width=.3285\linewidth]{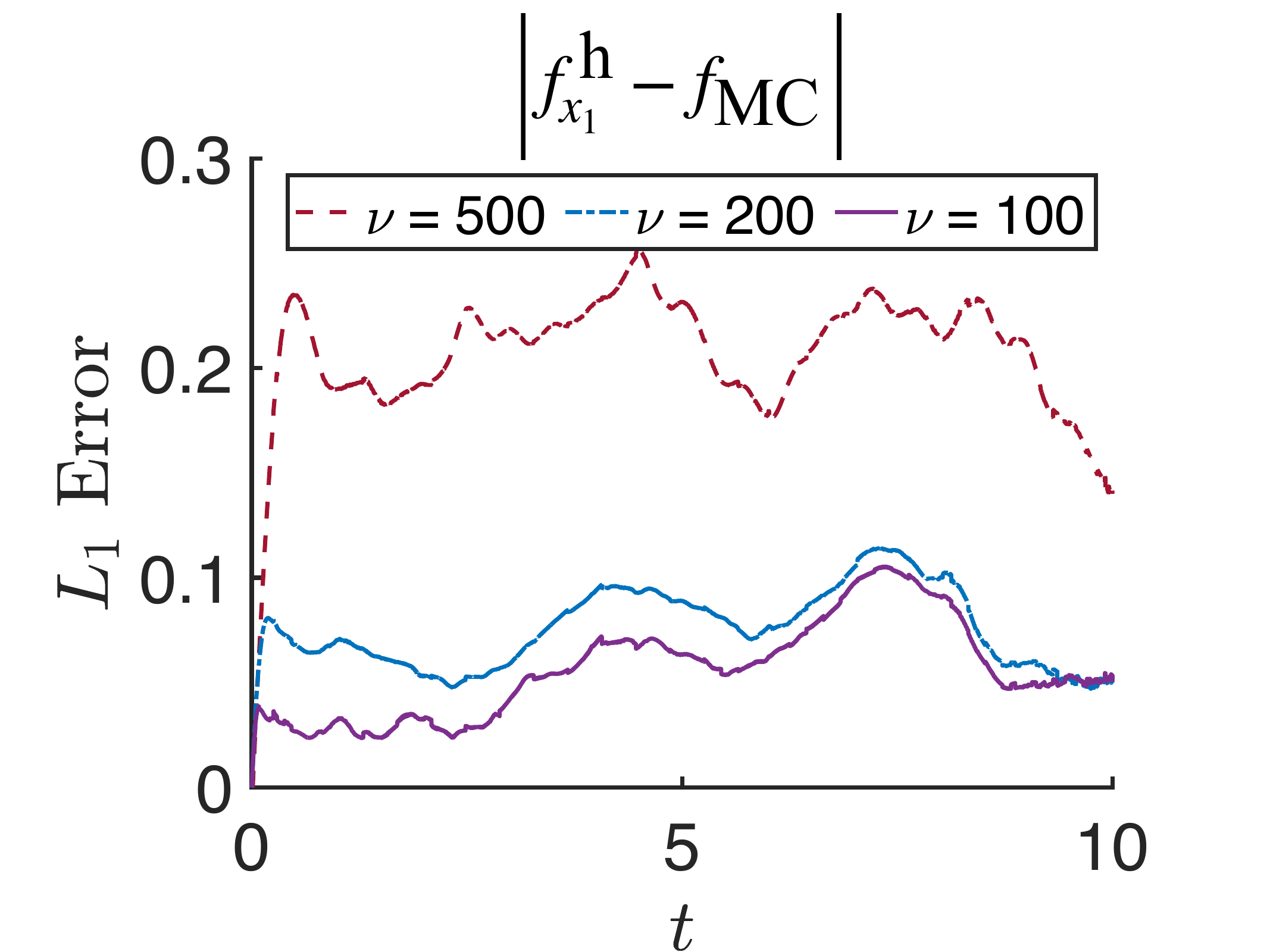}
    \includegraphics[width=.3285\linewidth]{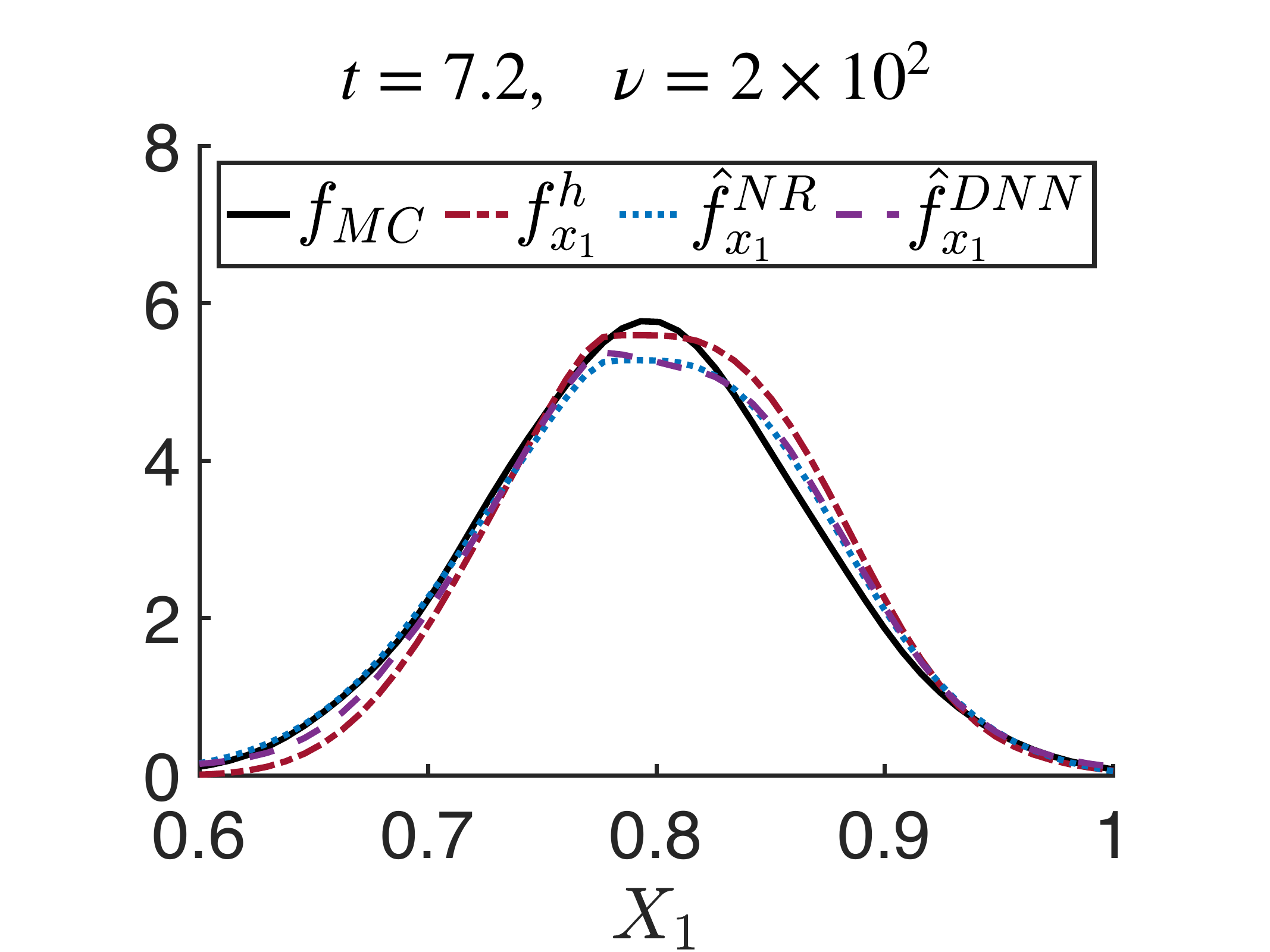}
    \includegraphics[width=.3285\linewidth]{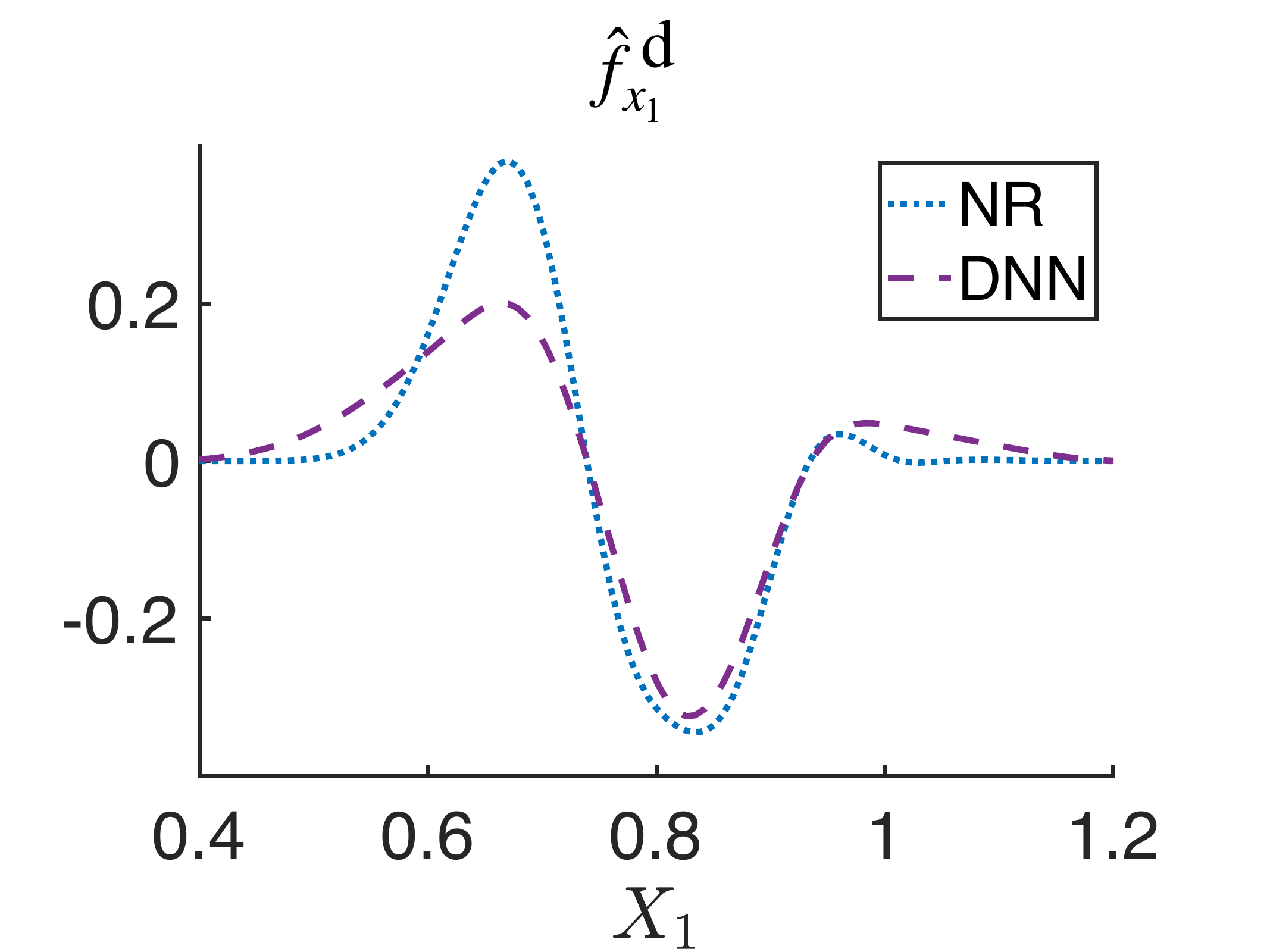}
\caption{(Left) Temporal evolution of $f_{x_1}^\text{h}$ $L_1$ error against the yardstick $f_\text{MC}$ for $\nu\in\{1,2,5\}\times 10^2$ and $N_\text{MC}^\text{tr} = 5\times 10^2$. (Middle) Snapshot of $f_\text{MC}$, $f_{x_1}^\text{h}$, $\hat f_{x_1}^\text{NR}$, and $\hat f_{x_1}^\text{DNN}$ for $\nu=2\times 10^2$ at time $t=7.2$. (Right) Defects $\hat{f}_{x_1}^\text{d}$ corresponding to the middle plot for DNN and NR observers. The latter is computed {\emph{ex post facto}} as $\hat f_{x_1}^\text{NR} - f_{x_1}^\text{h}$.}
\label{fig:lin_sol_snap}
\end{figure}

\Cref{fig:lin_sol_snap} (middle) reveals the RoPDF solutions to be overwhelmingly Gaussian. This is no surprise since \cref{eq:linrode} is linear with Gaussian noise and initial conditions. Moreover, the snapshot of $f_{x_1}^\text{h}$ for $\nu = 2\times10^2$ at time $t=7.2$ is a close match to the yardstick $f_\text{MC}$ even though the $L_1$ error is approximately $11\%$, as seen in left subfigure.  
Upon closer inspection, there are deviations in the mean, variance, and left tail from $f_\text{MC}$, though they appear minimal. This behavior with respect to $f_\text{MC}$ is similar at other times and for other combinations of $\nu$ and $N_\text{MC}^\text{tr}$ but is more pronounced as $\nu$ increases. Although it is pessimistic for PDFs, we limit our error to $L_1$ since theoretical NR convergence is established in this metric (see \Cref{app:nr} and \cite{Boulanger_2015}). However, regardless of the metric, there is always a sharp increase in $f_{x_k}^\text{h}$'s error at early times, where the error magnitude is largely determined by $\nu$. This is expected given the dynamics' initial transience, where the RoPDF quickly transitions away from the initial condition to the dominant periodic evolution seen in \Cref{fig:lin_sol} (right). Naturally, if $\nu$ is too large, even with $\hat\rmc$'s interpolation, the advection coefficient cannot properly account for this transience, and $f_{x_k}^\text{h}$ is perturbed away from the true dynamics without any means for correction, even if the coefficient is correctly estimated at later times. This is where our proposed assimilation methods pick up the slack.

\Cref{fig:lin_nr_err} (left) is the temporal $L_1$ error evolution of the NR observer against the observations used during assimilation, which helps visualize the NR procedure. The tick marks along the horizontal axis denote the relatively few time periods when $\lambda>0$ and RoPDF observations are assimilated into the dynamics via \cref{fnud}. They typically correspond to small magnitudes and sharp decreases in error, showing that observations are assimilated in quickly when it serves to increase predictive power. This figure also
shows the error associated with the NR (middle) and DNN (right) observers against $f_\text{MC}$, revealing that both approaches perform well compared to the homogeneous solution (\Cref{fig:lin_sol_snap}, left). The most striking result is that both observers, save for the initial transience, are relatively unaffected by temporal sparsity as long as $\nu$ is not unreasonably large. This fact can also be seen in our convergence rates in \Cref{fig:lin_scale}. Overall, the DNN slightly outperforms NR, which we contribute to relatively simple error distributions and defects. The latter can be seen in \Cref{fig:lin_sol_snap} (right). 
\begin{figure}[t!] 
	\centering  
  	\includegraphics[width=.3285\linewidth]{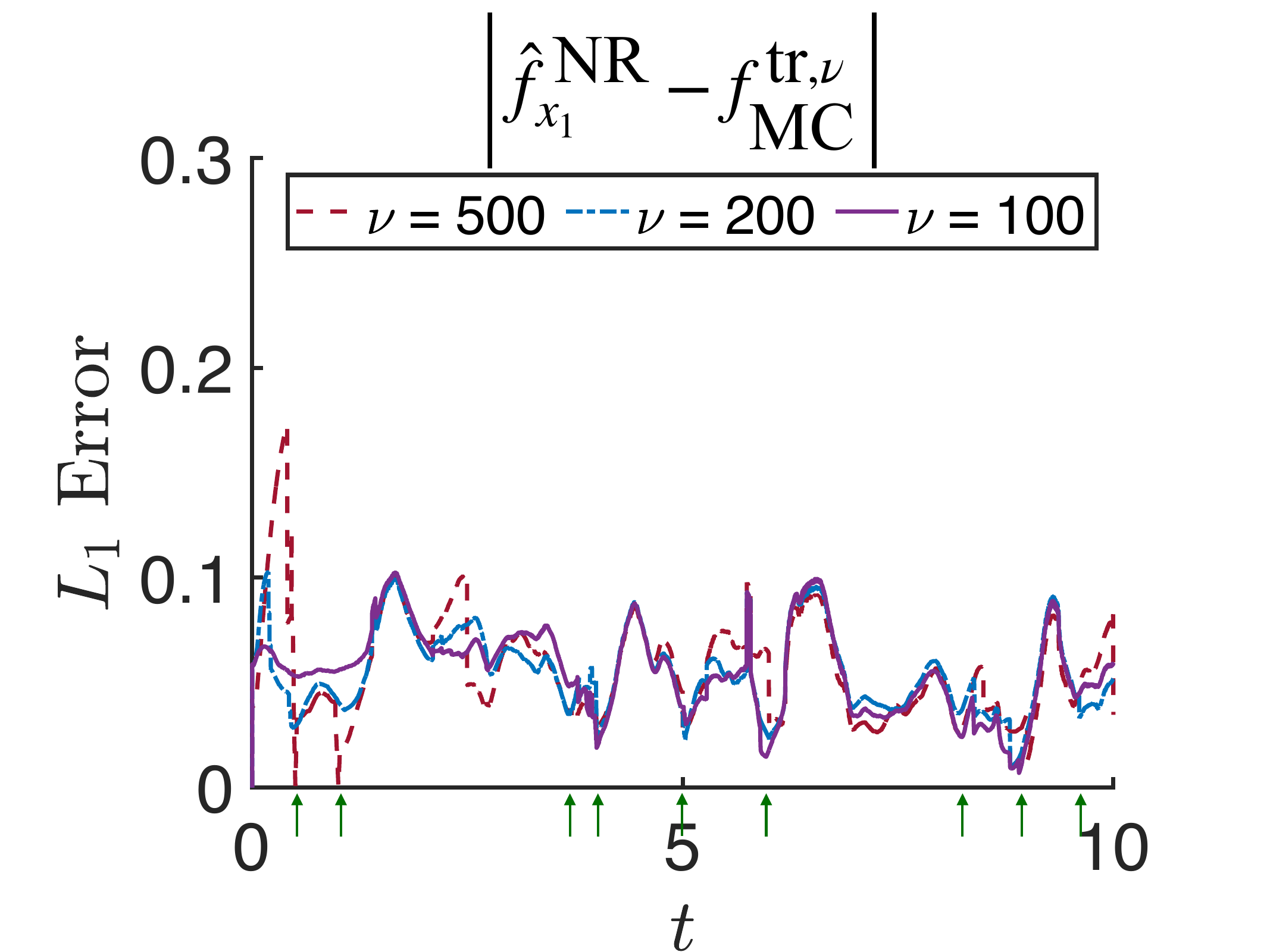}
   \includegraphics[width=.3285\linewidth]{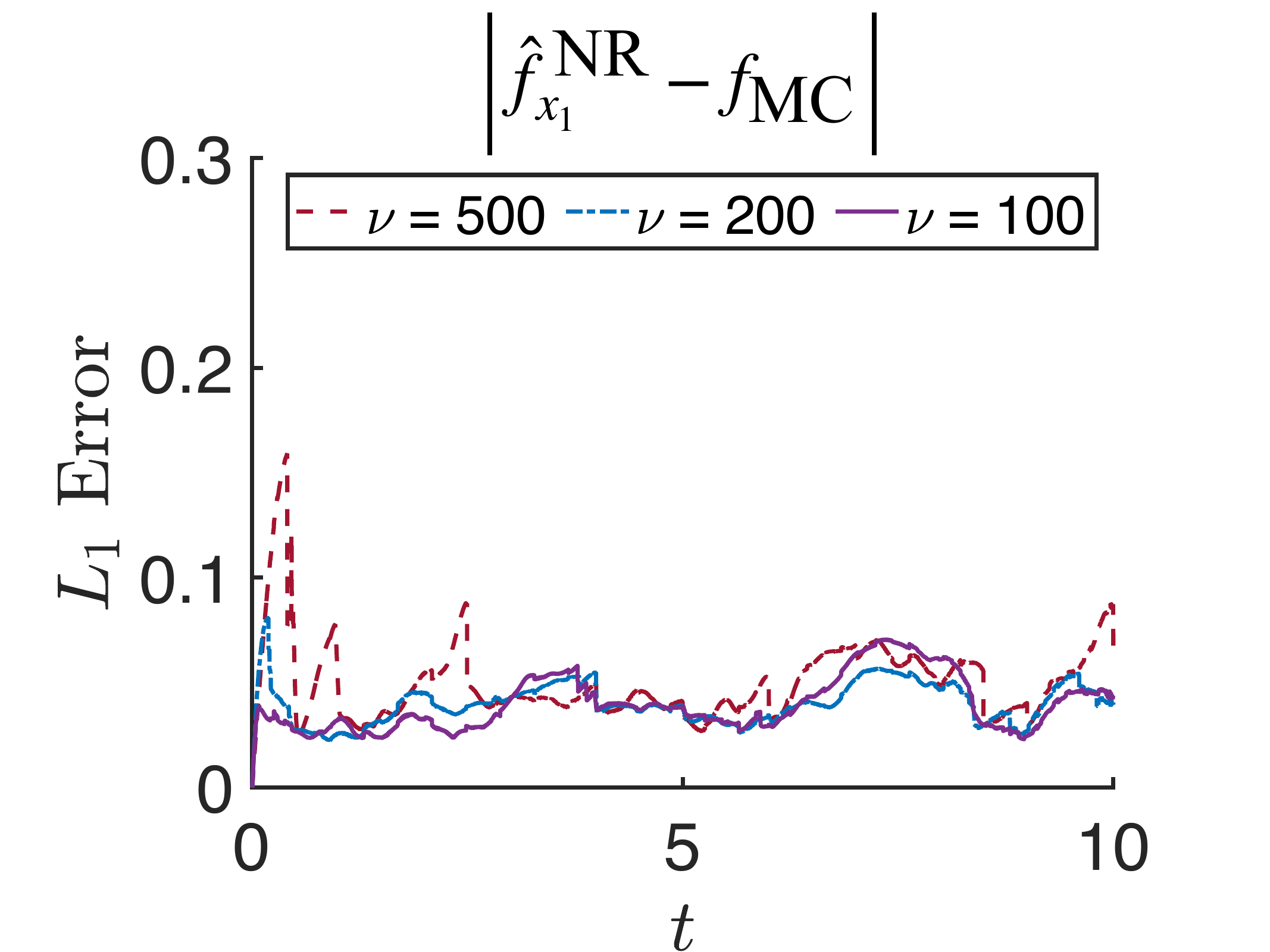}
   \includegraphics[width=.3285\linewidth]{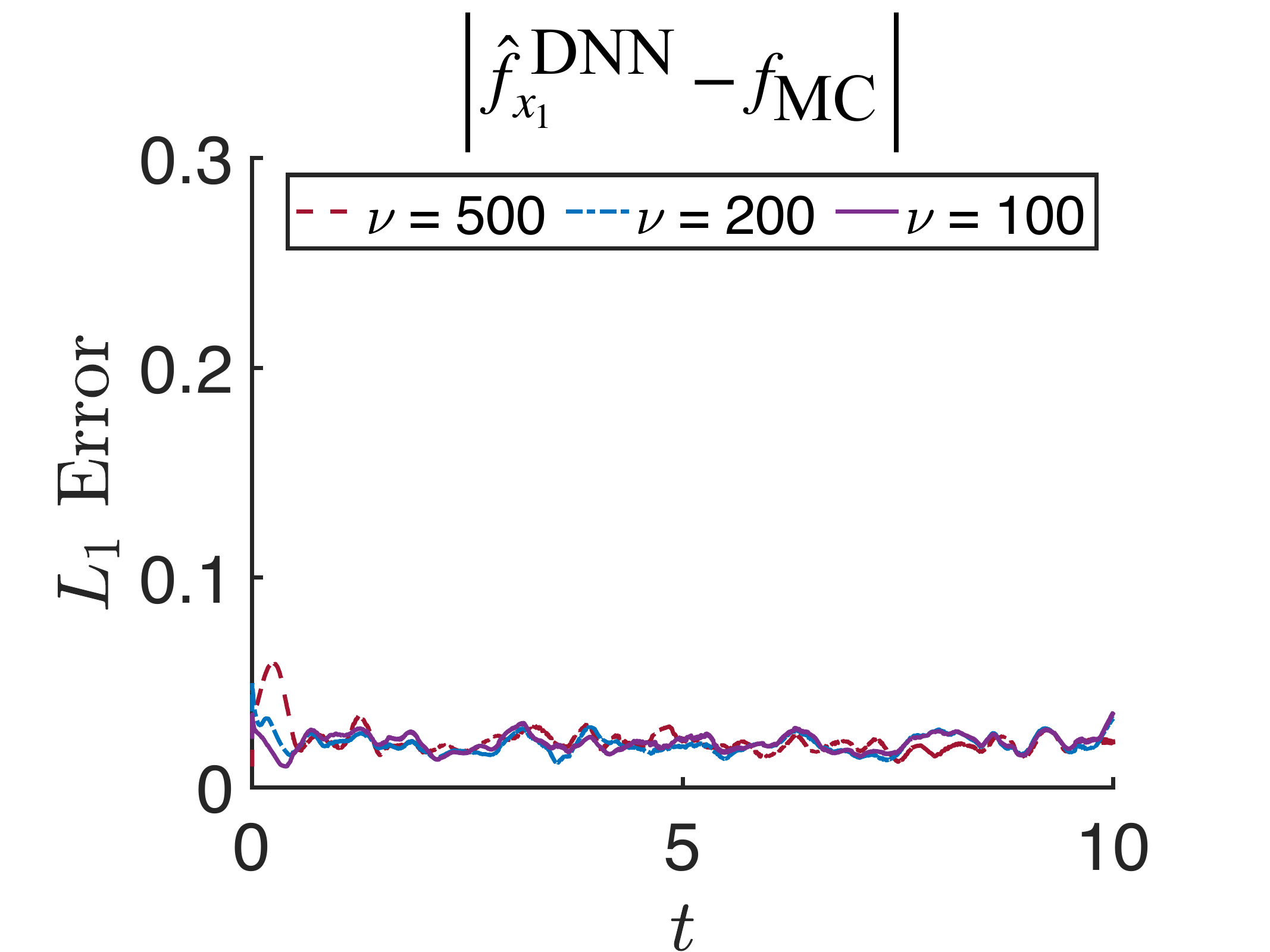}
\caption{Evolution of $L_1$ error for $\nu\in\{1,2,5\}\times 10^2$ and $N_\text{MC}^\text{tr} = 5\times 10^2$. (Left) $\hat f_{x_1}^\text{NR}$ against the assimilated $f_\text{MC}^{\text{tr},\nu}$. Green ticks on the $t$-axis represent short assimilation periods, i.e., when $\lambda_\nu(t)>0$. (Middle) $\hat f_{x_1}^\text{NR}$ against the yardstick $f_\text{MC}$. (Right) $\hat f_{x_1}^\text{DNN}$ against $f_\text{MC}$.}
\label{fig:lin_nr_err}
\end{figure}

\Cref{fig:lin_scale} provides convergence (in the normalized $L_1$ norm over space and time) of the assimilated observations (left), the NR observer (middle), and the DNN observer (right) as $N_\text{MC}^\text{tr}$ increases. For complete data ($\nu=1$), we recover the standard MC convergence rate of $\bo(1/\sqrt{N_\text{MC}^\text{tr}})$ for the observations. However, as data becomes sparse, this convergence considerably degrades due to the observations' construction via interpolation. For $\nu=1\times10^2$, the error of $f_\text{MC}^{\text{tr},\nu}$ increases in magnitude and the convergence slows to $\bo(1/\sqrt[3]{N_\text{MC}^\text{tr}})$. For $\nu>2\times10^2$, the error magnitude continues to increase and the rate is nearly constant. The NR observer, on the other hand, surpasses standard and quasi-MC rates with $\bo(1/N_\text{MC}^\text{tr})$ convergence. The remarkable feat is that this rate is nearly independent of $\nu$. This also holds for the DNN observer, but with slightly smaller magnitudes and sharper rates.

Overall, both approaches to assimilation are effective and cut costs of the MC approach by a factor of 4. This speedup is significant but not drastic given that the MC approach is on the scale of minutes in CPU time, which is due to linear dynamics and low dimensionalilty. Note, experiments were performed with an Apple M2 Max chip in parallel on $12$ CPU cores. 
\begin{figure}[h!] 
	\centering  
  	\includegraphics[width=.3285\linewidth]{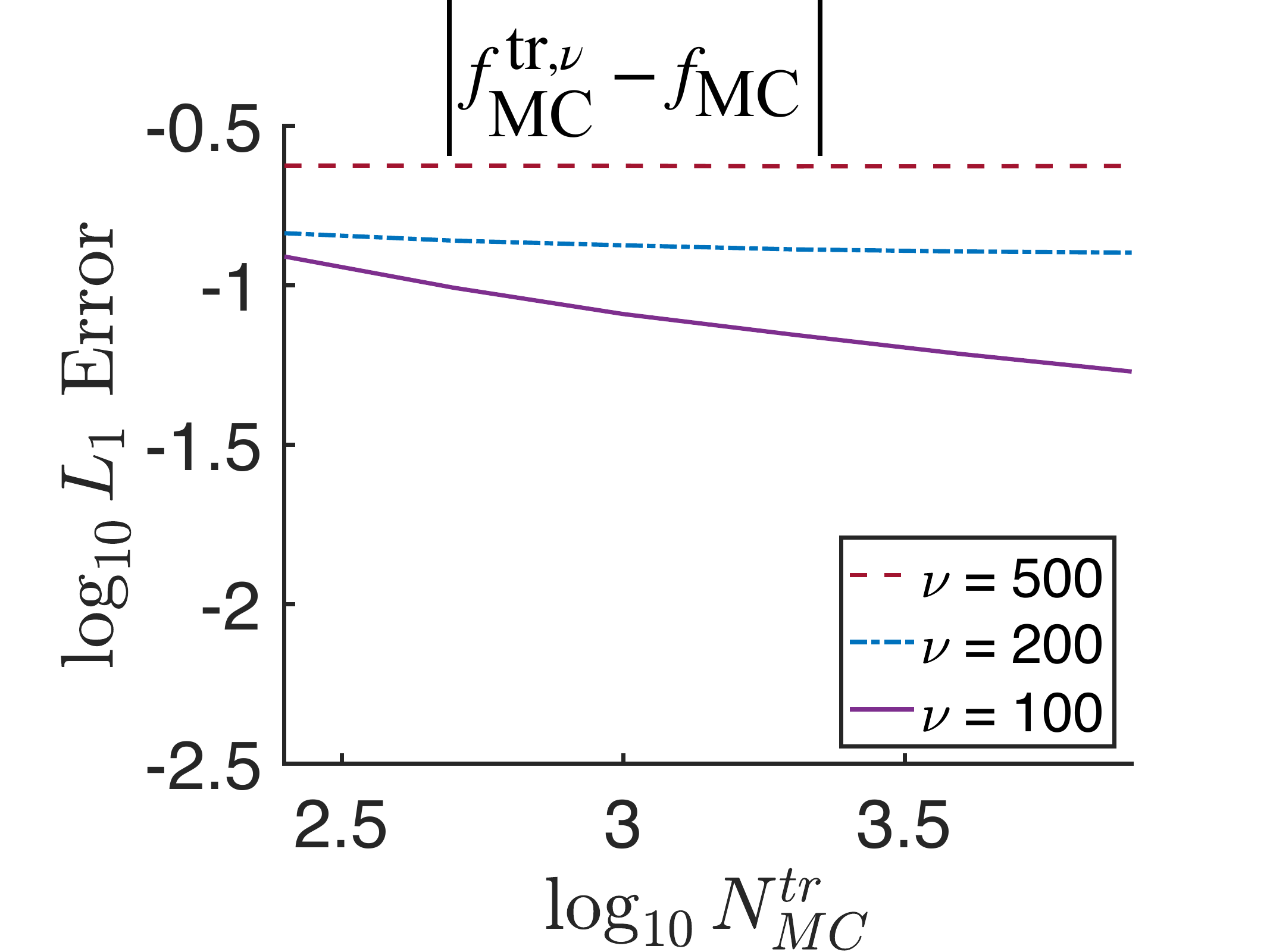}
    \includegraphics[width=.3285\linewidth]{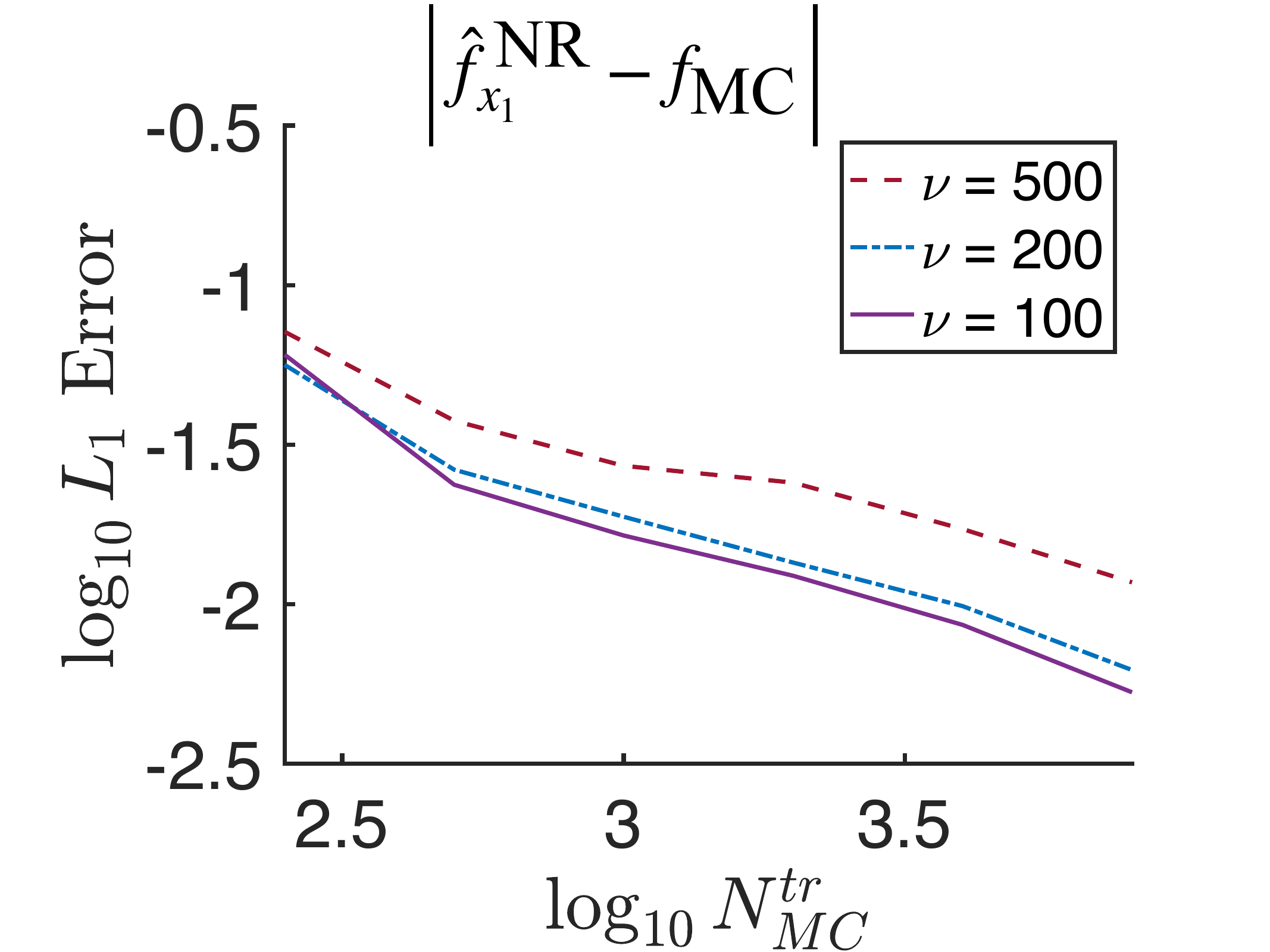}
    \includegraphics[width=.3285\linewidth]{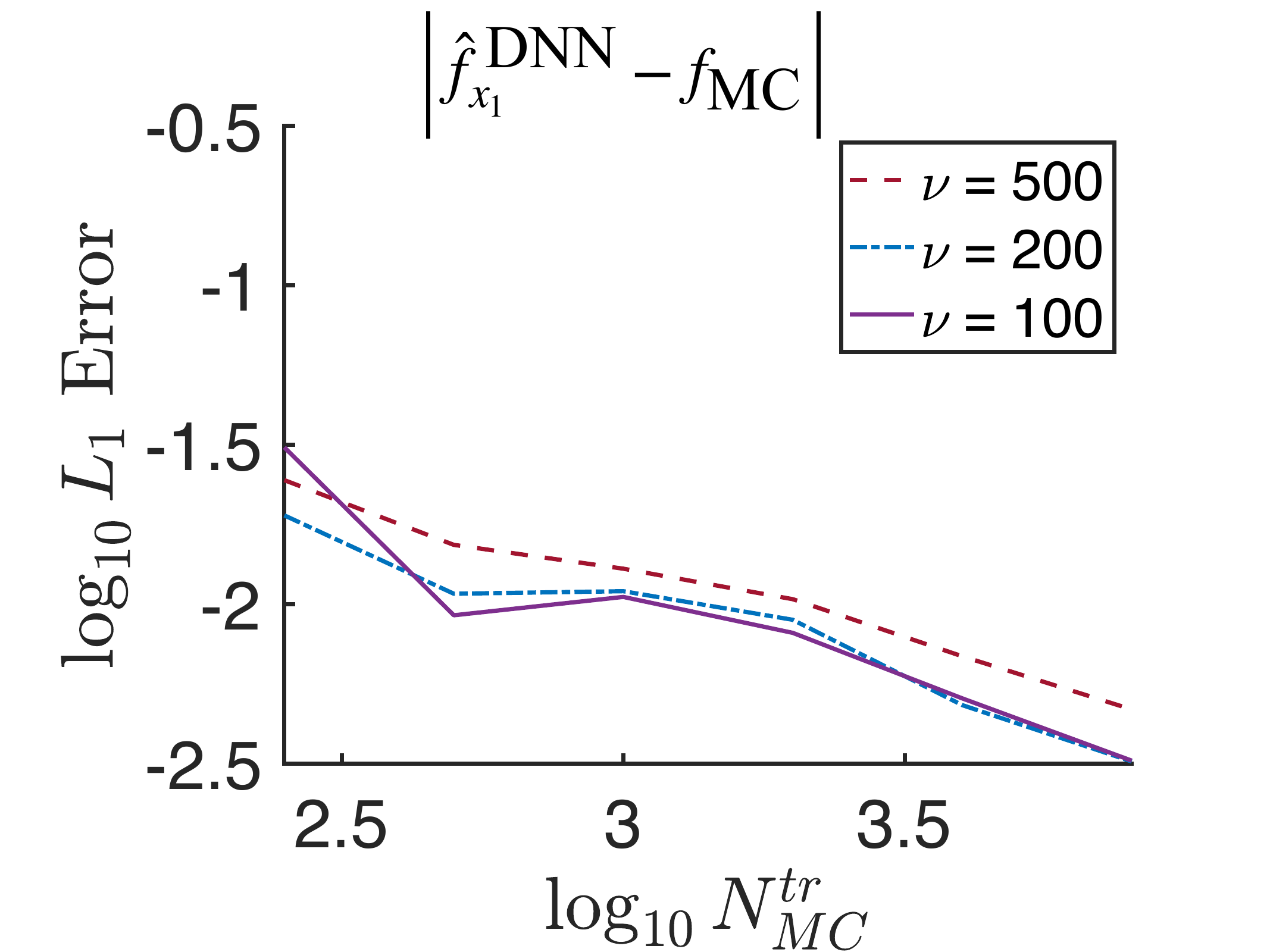}
\caption{Convergence rates, on a log-log scale, for the spatiotemporal $L_1$ error of the observations $f_\text{MC}^{\text{tr},\nu}$ (left), $\hat{f}_{x_1}^\text{NR}$ (middle), and $\hat{f}_{x_1}^\text{DNN}$ (right) for $\nu=\{1,2,5\}\times10^2$ as $N_\text{MC}^\text{tr}$ increases.}
\label{fig:lin_scale}
\end{figure}


\subsection{Power System Cascade Outages}
\label{sec:power}

We study how our method applies to electrical power systems, particularly in characterizing cascading failure modes dependent on stochastic sources. As the grid sees an increase in renewable generation sources and electric vehicles, the risk of stochastic fluctuations triggering a cascade of failures rises, potentially leading to significant economic impacts and safety risks.

The power system consists of $N_\text{bus}=N_\text{g}+N_\text{l}$ buses, comprised of $N_\text{g}$-many generators and $N_\text{l}$-many loads, and a network of $N_\text{line}$-many transmission lines. Sudden perturbations of the steady state can lead to the overloading of transmission lines, resulting in their sudden trip or disconnection. This disconnection may trigger further line disconnections in a cascade fashion. Therefore, to characterize the risk of cascades, the RoPDF method is employed, which can be used to compute the probability of line outages in response to stochastic perturbations.

To address this issue in a computationally tractable fashion, Zheng and DeMarco proposed a port-Hamiltonian model to represent the potential of such cascading outages that incorporates line tripping by means of ``smooth bistable" variables \cite{DeMarco_1987, Zheng_2015phd}. We show the complete model:
\begin{align} \label{pham2}
	\dot{\boldsymbol\w}_g &= -\mathbf M_g^{-1}\mathbf D_g\boldsymbol\w_g - 
	\mathbf M_g^{-1}\U_1^\top\mathbf f(\a,\bv_l,\boldsymbol\gamma), \notag \\
	\dot\a &= \U_1\boldsymbol\w_g -\left[\U_2\mathbf D_l^{-1}
	\U_2^\top\right]\mathbf f(\a,\bv_l,\boldsymbol\gamma), \notag \\
	\dot{\mathbf{V}}_l &= -\mathbf D_v^{-1}\mathbf g(\a,\mathbf V_l,\boldsymbol\gamma), \notag \\
	\dot{\boldsymbol\gamma} &= -\mathbf D_\g^{-1} 
	\mathbf h(\a,\mathbf V_l,\boldsymbol\gamma),
\end{align}
where
\begin{align*}
	\U &\triangleq [-\mathbf e\,|\,\mathbf I_{N_\text{bus}}] = [\U_1\,|\,\U_2],\,\,\U_1\in\R^{N_\text{bus}
	\times(N_\text{g}+1)},\,\,\U_2\in\R^{N_\text{bus}\times N_\text{l}}, 
	\mathbf e \triangleq [1,\dots,1]^\top\in\R^{N_\text{bus}}, 
\end{align*}
and $\mathbf I_{N_\text{bus}}$ is the $N_\text{bus}\times N_\text{bus}$ identity matrix. The system states $\x(t)\triangleq \left[\boldsymbol\w_g^\top, \a^\top, \bv_l^\top, \boldsymbol\g^\top\right]^\top$ are comprised of $(N_\text{g}+1)$ generator speeds ${\boldsymbol\w}_g$ (including a non-physical reference/slack bus), $N_\text{bus}$ non-slack angles $\a$, $N_\text{l}$ load voltage magnitudes $\mathbf V_l$, and $N_\text{line}$ indicator-like bistable variables $\boldsymbol\gamma$ representing the operating status of each line. Hence, \cref{pham2} is an $(N_\text{g}+1+N_\text{bus}+N_\text{l}+N_\text{line})$-dimensional system. Here, $\mathbf M_g\in\R^{(N_\text{g}+1)\times (N_\text{g}+1)}$ is the  generator mass/inertia matrix and $\mathbf D_g\in\R^{(N_\text{g}+1)\times (N_\text{g}+1)}$, $\mathbf D_l\in\R^{N_\text{l}\times N_\text{l}}$, $\mathbf D_v\in\R^{N_\text{l}\times N_\text{l}}$, and $\mathbf D_\g\in\R^{N_\text{line}\times N_\text{line}}$ are the states' various damping matrices. $\mathbf f(\a,\bv_l,\boldsymbol\gamma)\in\R^{N_\text{bus}}$ represents the net power at each of the non-slack buses, where the sign convention takes absorbed power as positive. In other words,
\[ f_i(\a, \bv_l,\boldsymbol\gamma) \triangleq \tilde f_i(\a, \bv_l, \boldsymbol\gamma) - P_i^0, \qquad i\in\{1,\dots,N_\text{bus}\}, \]
where $\mathbf P^0 = [P_1^0, \dots, P_{N_\text{bus}}^0]^\top\in\R^{N_\text{bus}}$
represents the prescribed active mechanical power from the generators and the negative active load power demands. $\mathbf Q^0\in\R^{N_\text{l}}$ is the reactive power analog of $\mathbf P^0$, but defined only at the loads, and $\mathbf g(\a,\bv_l,\boldsymbol\gamma)\in\R^{N_\text{l}}$ is defined as 
\[ g_i(\a, \bv_l,\boldsymbol\gamma) \triangleq V_{l,i}^{-1}\left(\tilde g_i(\a, \bv_l,\boldsymbol\gamma) - Q_i^0\right), \qquad i\in\{N_g+1,N_g+2,\dots,N_\text{bus}\}. \]

The system \cref{pham2} approximates tripping dynamics via the bistable states $\boldsymbol\gamma$ and their corresponding velocity field
\begin{align} \label{eqh}
	\mathbf h(\a,\mathbf V_l,\boldsymbol\gamma) \triangleq \tilde{\mathbf{h}}(\a,\mathbf V_l)
	- \mathbf H\odot\boldsymbol\theta(\boldsymbol\gamma).
\end{align}
The smooth thresholding function
\begin{align} \label{eqth}
	\theta_k(\boldsymbol\gamma) \triangleq 
	2\left[-\exp(-20\g_k)+\exp(-200\g_k)+\exp(20(\g_k-1))-\exp(200(\g_k-1))\right],
\end{align}
for $k\in\{1,\dots,N_\text{line}\}$, is constructed so that upon integrating \cref{pham2}, two potential wells are created very close to zero and one, where the latter has height $\approx H_k$. When the line energy $\tilde{h}_k$ (see \cite[Eq. 3.12]{Zheng_2015phd}) exceeds the threshold $H_k$, $h_k(\a,\mathbf V_l,\boldsymbol\gamma)$ becomes very large, driving $\g_k$ in \cref{pham2} quickly to zero, effectively removing the line from the system. Moreover, due to \cref{eqth}, once $\g_k$ transitions to zero, it stays there, save for small fluctuations around zero.

To account for stochastic fluctuations at the loads, we add $N_p=2N_\text{l}$ OU noise processes $\boldsymbol\xi(t)$ to $\mathbf P^0$ and $\mathbf Q^0$ such that 
\begin{align*}
	f_i(\a, \bv_l,\boldsymbol\g, \boldsymbol\xi) &\triangleq 
	\tilde f_i(\a, \bv_l,\boldsymbol\g) - P_i^0 - \s_{\text{P},i}\xi_{\text{P},i}, \\
	g_i(\a, \bv_l,\boldsymbol\g, \boldsymbol\xi) &\triangleq 
	 V_{l,i}^{-1}\left(\tilde g_i(\a, \bv_l,\boldsymbol\g) - Q_i^0 - \s_{\text{Q},i}
	 \xi_{\text{Q},i}\right), \qquad i\in\{N_\text{g}+1,N_\text{g}+2,\dots,N_\text{bus}\},
\end{align*}
where the components of $\boldsymbol\xi(t) \triangleq \left[ \boldsymbol\xi_\text{P}^\top(t), \boldsymbol\xi_\text{Q}^\top(t)\right]^\top$ are defined by \cref{eq:ou2} and taken to be uncorrelated. We take all noise processes to have identical correlation length of $\tau=10^{-2}$ and set all $\s_{\text{P},i} \approx 2.19$ and $\s_{\text{Q},i}\approx 1.55$. For our experiments below, this puts the RODE in the high-noise regime. Many methodologies that use large deviation arguments to obtain asymptotic transmission failure rates, which typically require the existence of a nice closed-form stationary measure such as a Gibbs measure, usually do not perform well in this setting \cite{Roth_2021}.

All experiments that follow are over the time interval $[0,T_f]$ with $T_f=0.5$ for the IEEE 14-Bus System, giving a 47-dimensional RODE system.
The random initial conditions of the RODE are computed in the same manner as in \cite{MaltbaPES}. That is, an equilibrium point of the deterministic power system is found by solving the optimal power flow via \texttt{MATPOWER} \cite{Zimmerman2011}. The equilibrium point is treated as a deterministic initial condition for the RODE, which is burned in via $N_\text{MC}^\text{tr}$ MC simulations over the entire time horizon. During this burn-in, all tripping thresholds are set to $\mathbf H \equiv 1$ (equivalent to a line rating of $200$ megavolt amperes) so that no lines are tripped. The resulting samples of $\x(T_f)$ are then treated as independent samples of the random initial condition $\x^0$ at time $t=0$, which are used in generating MC realizations of the RODE system over the time interval $(0,T_f]$ as well as post-processed with KDE \cite{Botev_2010} to compute the RoPDF for the QoI at $t=0$. After the noise burn-in period, we perturb the system out of its quasi-equilibrium by manually removing line $15$ at time $t=0$. Additionally, at time $t=0$, we reduce the thresholds for lines $12$ and $17$ to $H_{12} = 0.0135$ and $H_{17} = 0.0125$, respectively, to mimic so-called ``weak lines," which cannot afford normal load flow, that occur in physical power systems under various circumstances, e.g., bad weather.

Following \cite{Zheng_2015phd, Zheng3} (see Table II in the latter), we set $\text{diag}(\mathbf M_g) = 5.3\times10^{-2}$, $\text{diag}(\mathbf D_g) = 5\times 10^{-2}$, $\text{diag}(\mathbf D_l) = 5\times10^{-3}$, and $\text{diag}(\mathbf D_v) = 10^{-2}$, which are the same parameter choices for the experiments in \cite{Roth_2021}. We determined (via convergence studies) that $\text{diag}(\mathbf D_\g) = 10^{-3}$ is the largest possible value that achieves realistic tripping dynamics.


\subsubsection{RoPDF Equation \& Numerics} \label{sec:pownum}

Since we are interested in quantifying the uncertainty concerning line failures in the power grid, we consider the real-valued QoI to be $z(\x(t)) = \g_k(t)$, which represents the operational status of the $k$-th power line. Since the phase space of $\g_k$ is technically unbounded, we let $Z_k\in \R$ represent a variable in its phase space. Following the derivation in \Cref{sec:main}, the exact RoPDF equation for the marginal PDF $f_{\g_k}(Z_k;t)$ of $\g_k(t)$ is given by
\begin{align} \label{gpdfeq}
	&\frac{\d f_{\g_k}}{\d t} + \frac{\d}{\d Z_k}\Big( -D_{\g,kk}^{-1}\left(
	\rmc(Z_k,t) - H_k\theta_k(Z_k)\right)\Big) = 0, \notag \\
	& f_{\g_k}(Z_k;0) = f_{\g_k}^0(Z_k),
\end{align}
with vanishing boundary conditions, where the regression function is the conditional expectation
$\rmc(Z_k,t)\triangleq\left\langle\tilde h_k(\a,\bv_l)\, \big | \,\g_k(t)=Z_k\right\rangle$.
Since the thresholding function $\theta_k$ depends only on the QoI $\g_k$, the advection coefficient in \cref{gpdfeq} is partially separable, and thus $\theta_k(Z_k)$ has been pulled out of the conditional expectation. In the experiments that follow,  $\rmc$ is always estimated by $\hat\rmc$ via GLLR for each $t_{m_l}\in\mathbf T_\nu$. In all MC simulations, three lines underwent tripping dynamics, including both weak lines. Out of these three, the RoPDF for line $12$ had the most complex dynamics. Hence, we limit our presentation to the $k=12$ case.

As seen in \Cref{fig:pow_sol} and \Cref{fig:powpdfs}, the dynamics of \cref{gpdfeq} transition the RoPDF from unimodal to bimodal, where the essential support of the modes is quite small. To accurately capture these dynamics, we take the spatial mesh $\Z_{12}$ to be fixed but nonuniform with $\Delta Z_{12}$ ranging from $10^{-4}$ near the mode locations $Z_{12}\approx0$ and $1$ to $5\times 10^{-2}$ in between the modes, resulting in approximately $850$ grid cells. If uniform time stepping is used for the Lax-Wendroff discretization, the CFL condition requires $\Delta t =10^{-6}$. Even though variable time stepping and/or different PDE discretizations may be used to reduced the number of time steps, leaving $\Delta t$ uniform in our discretization serves to demonstrate one way in which temporal sparsity can arise. Another comes from the MC simulations and the RODE discretization. For the given stiff power system, an explicit RODE discretization would require time steps as small as $10^{-9}$ to ensure stability. Our strong-order implicit time stepping can be taken much larger, but a small
$\Delta t=10^{-5}$ to $10^{-4}$ is still required to accurately capture quick transitions during tripping dynamics. However, to reduce memory requirements, we only store the MC training samples at time increments of $10^{-3}$. Hence, the sparsity factor with respect to the RODE discretization is $10$ to $10^2$ but is $10^3$ compared to the PDE discretization. Given our choice of notation, our sparsity factor $\nu$ refers to the latter, i.e., $\nu=10^{3}$. For this application, we do not consider additional sparsity factors since $\nu=10^3$ is considerably large for the given dynamics, and it has arisen naturally due to memory limitations. Similar to our presentation in \Cref{sec:lin}, we limit our head-to-head comparisons to a single sample size of $N_\text{MC}^\text{tr} = 2\times10^3$, but vary this samples size to determine overall convergence rates. The total number of MC trials required to compute the yardstick solution $f_\text{MC}$ for $\g_{12}(t)$ is $N_\text{MC}=10^5$.

 \begin{figure}[h!] 
	\centering  
	\begin{subfigure}{.33\textwidth}
  		\centering
  		\includegraphics[width=1\linewidth]{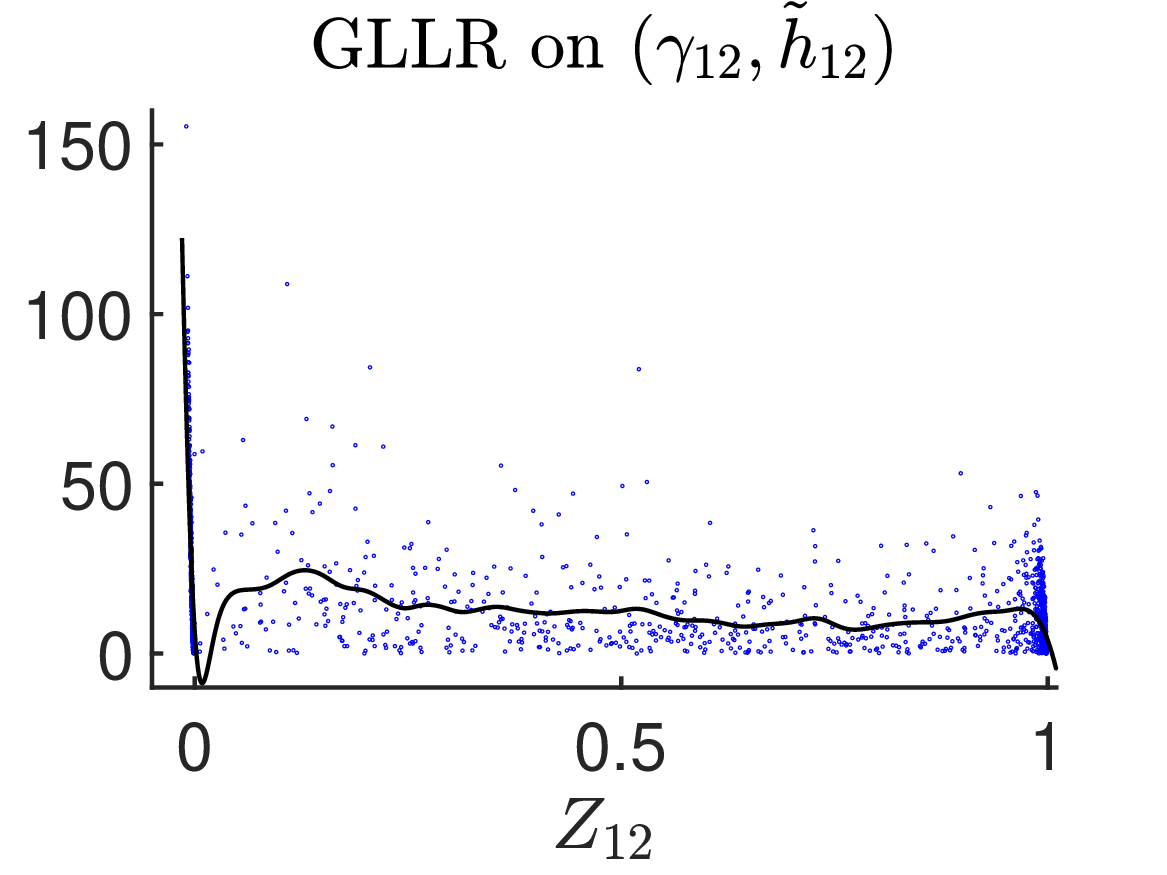}
	\end{subfigure}%
	\begin{subfigure}{.33\textwidth}
 		 \centering
  		\includegraphics[width=1\linewidth]{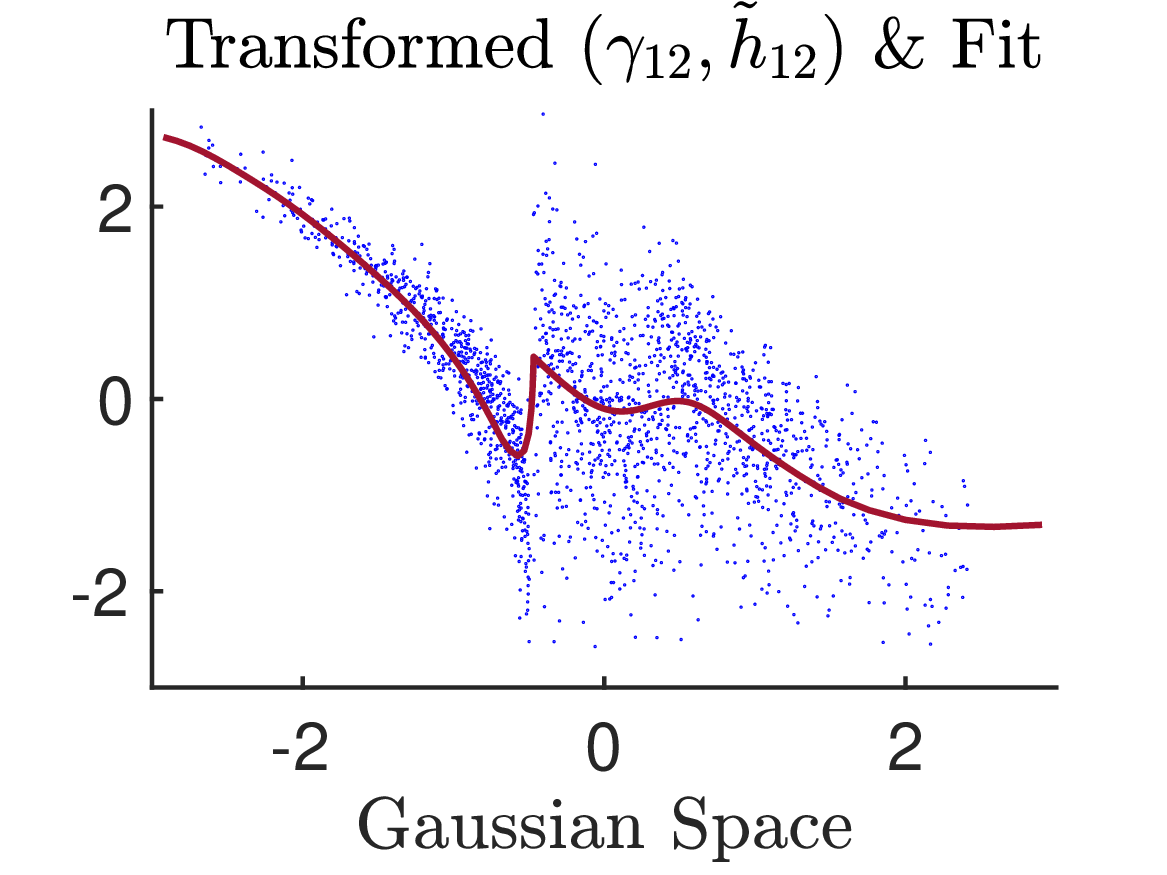} 
	\end{subfigure}%
	\begin{subfigure}{.33\textwidth}
 		\centering
  		\includegraphics[width=1\linewidth]{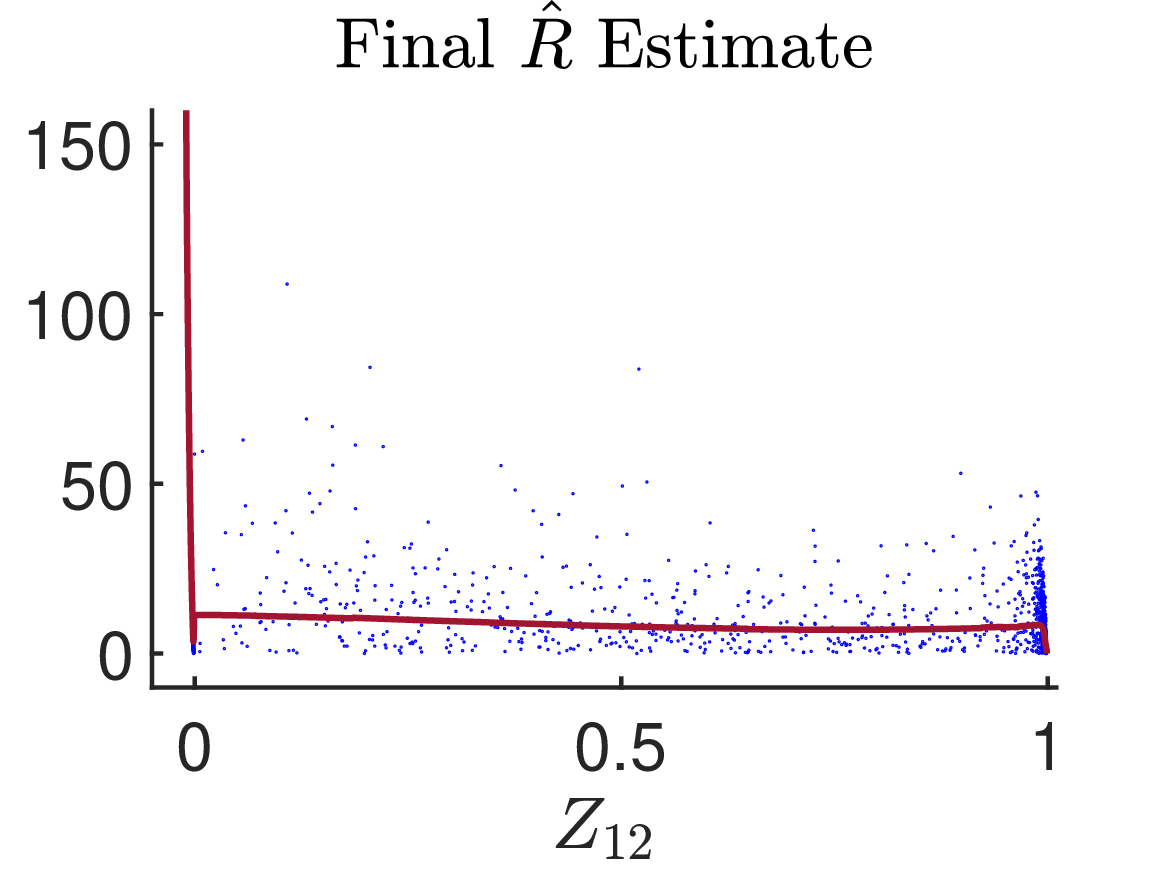}
  	\end{subfigure}
\caption{$N_\text{MC}^\text{tr} = 2\times 10^3$ MC realizations (blue dots) of $(\gamma_{12}, \tilde{h}_{12})$ at time $t=0.1$. (Left) GLLR estimate of $\hat\rmc$ using $10$-fold CV for bandwidth selection. (Middle) Data transformed to standard Gaussian variates and corresponding GLLR fit with simple plug-in bandwidth. (Right) Fit from the middle plot transformed back to the original scale to obtain a more accurate estimate of $\hat\rmc$.}
\label{fig:transform}
\end{figure}

\subsubsection{Regression} \label{sec:reg}

Regarding the regression estimates $\hat \rmc$, when the $k$-th line in the system is tripped and $\g_k$ is the QoI, the underlying dynamics make regression on the MC sample data difficult. However, at any given time, the response (line energy) data $\tilde h_k(\a,\bv_l)$ is always nonnegative and nicely right-skewed, allowing these variates to be efficiently transformed into standard normal variates via the one-parameter Box-Cox transformation. The transformation parameter is selected via maximum likelihood estimation with the Shapiro-Wilk goodness-of-fit test as was done in \cite{Asar}. The qualitative behavior of the underlying predictor data $\g_k$ associated with $\hat\rmc$ drastically changes during the transition period after the line is tripped, and therefore no single parametric transformation can be expected to perform well. Since we must already compute the observations $f_\text{MC}^{\text{tr},\nu}$ for the assimilation procedures, we can easily convert these PDFs into CDFs via inexpensive quadrature. Evaluating these CDFs at the $N_\text{MC}^\text{tr}$-many variates of $\g_k$ via 1D interpolation gives approximately uniform variates on $[0,1]$. Applying the inverse standard normal CDF to these variates gives approximately standard Gaussian variates. Given that the predictor and response data have each been transformed to (univariate) standard Gaussian variates, an optimal plug-in bandwidth estimator for Gaussian data can be employed for GLLR \cite[Ch. 3]{Bowman1997}, avoiding costly CV procedures. The optimal, robust estimator is given by $\hat s_k\triangleq \sqrt{\hat s\{\g_k\}\hat s\{\tilde h_k\}}$, where $\hat s\{\cdot\}$ is defined as 
\begin{equation} \label{eq:bw}
	\hat s\{y\} \triangleq \left(\frac{4}{3N_\text{MC}^\text{tr}}\right)^{0.2}
	\text{med}(|y-\text{med}(y)|)\,/\,0.6745.
\end{equation}

\Cref{fig:transform} (left) displays $N_\text{MC}^\text{tr} = 2\times 10^3$ MC realizations (blue dots) of $(\gamma_{12}, \tilde{h}_{12})$ at time $t=0.1$. On this original scale, the data near $Z_{12}\approx0$ (the RoPDF's left mode) nearly forms a vertical line. In order to remotely capture this behavior with GLLR, $\hat\rmc$ becomes negative and too rough between the modes, even with $10$-fold CV for bandwidth selection (black curve). However, after applying the transformations and performing GLLR with the much cheaper plug-in bandwidth (middle), the fit becomes excellent. The inverse transforms are applied to obtain a much cheaper and more accurate $\hat\rmc$ (right). Its temporal evolution, in addition to the full advection coefficient, is given in \Cref{fig:pow_sol}.

\begin{figure}[h!] 
	\centering
	\begin{subfigure}{.33\textwidth}
  		\centering
  		\includegraphics[width=1\linewidth]{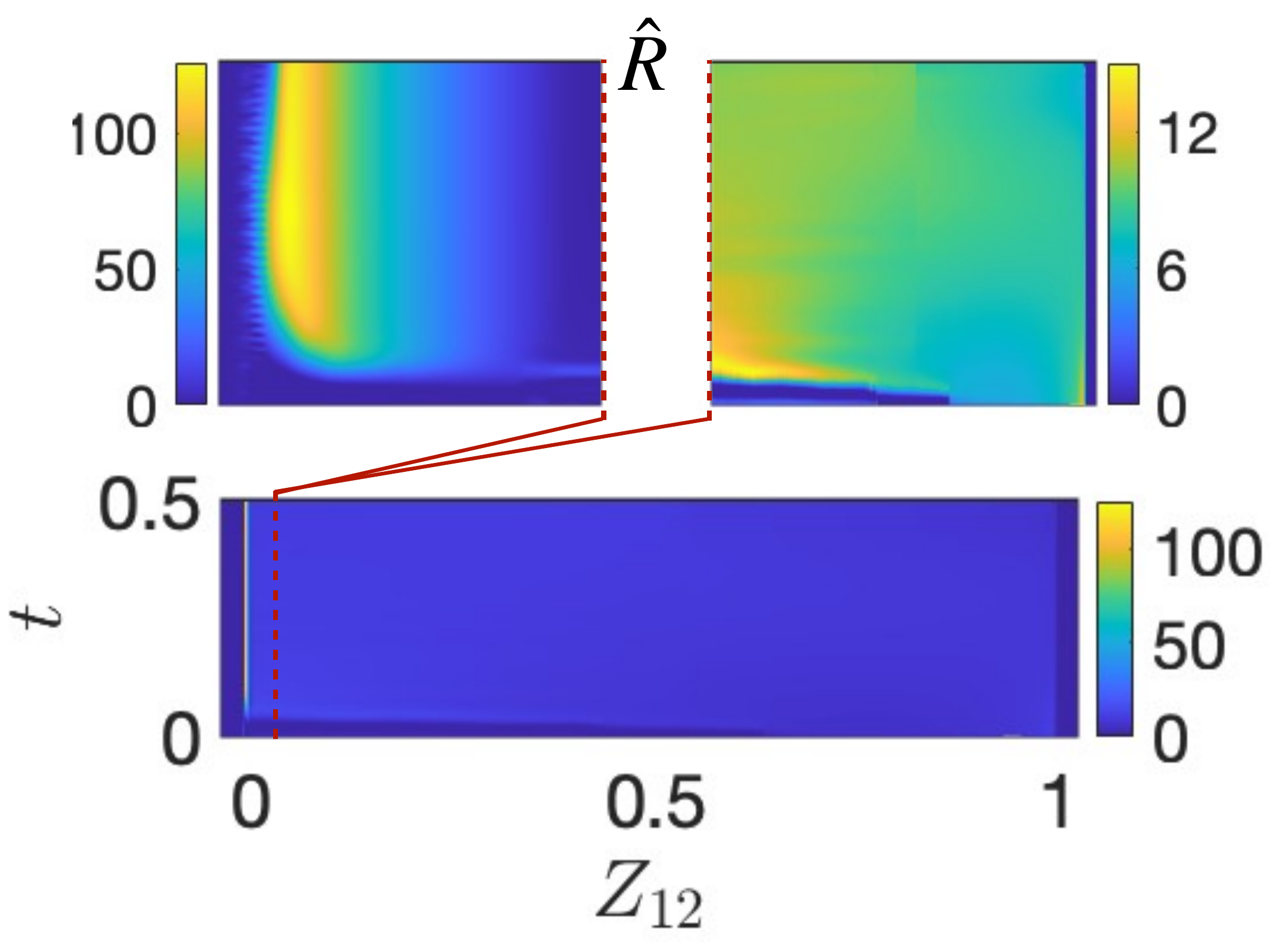}
	\end{subfigure}%
	\begin{subfigure}{.33\textwidth}
 	 	\centering
 		 \includegraphics[width=1\linewidth]{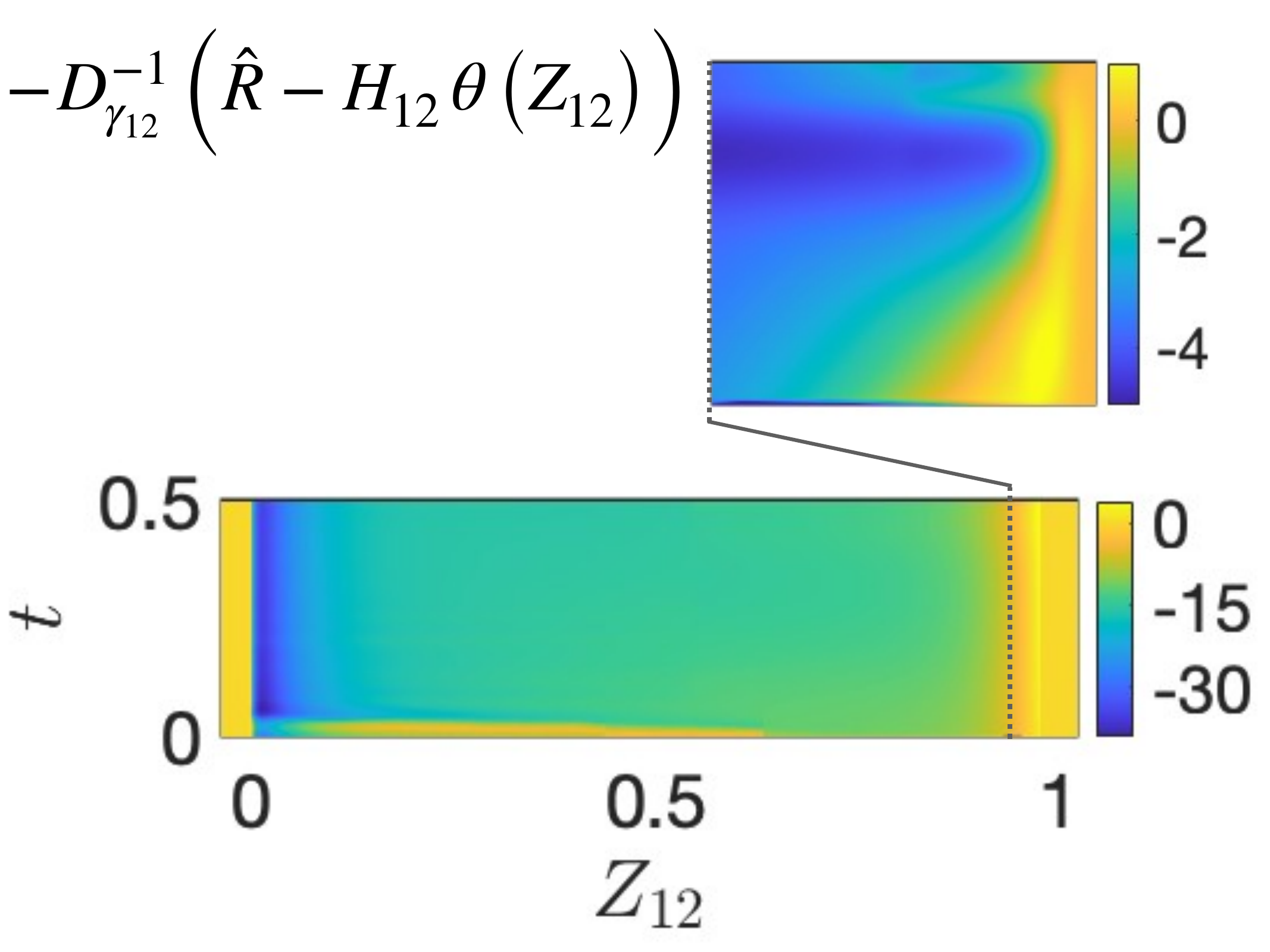}
	\end{subfigure}%
	\begin{subfigure}{.33\textwidth}
  		\centering
 	 \includegraphics[width=1\linewidth]{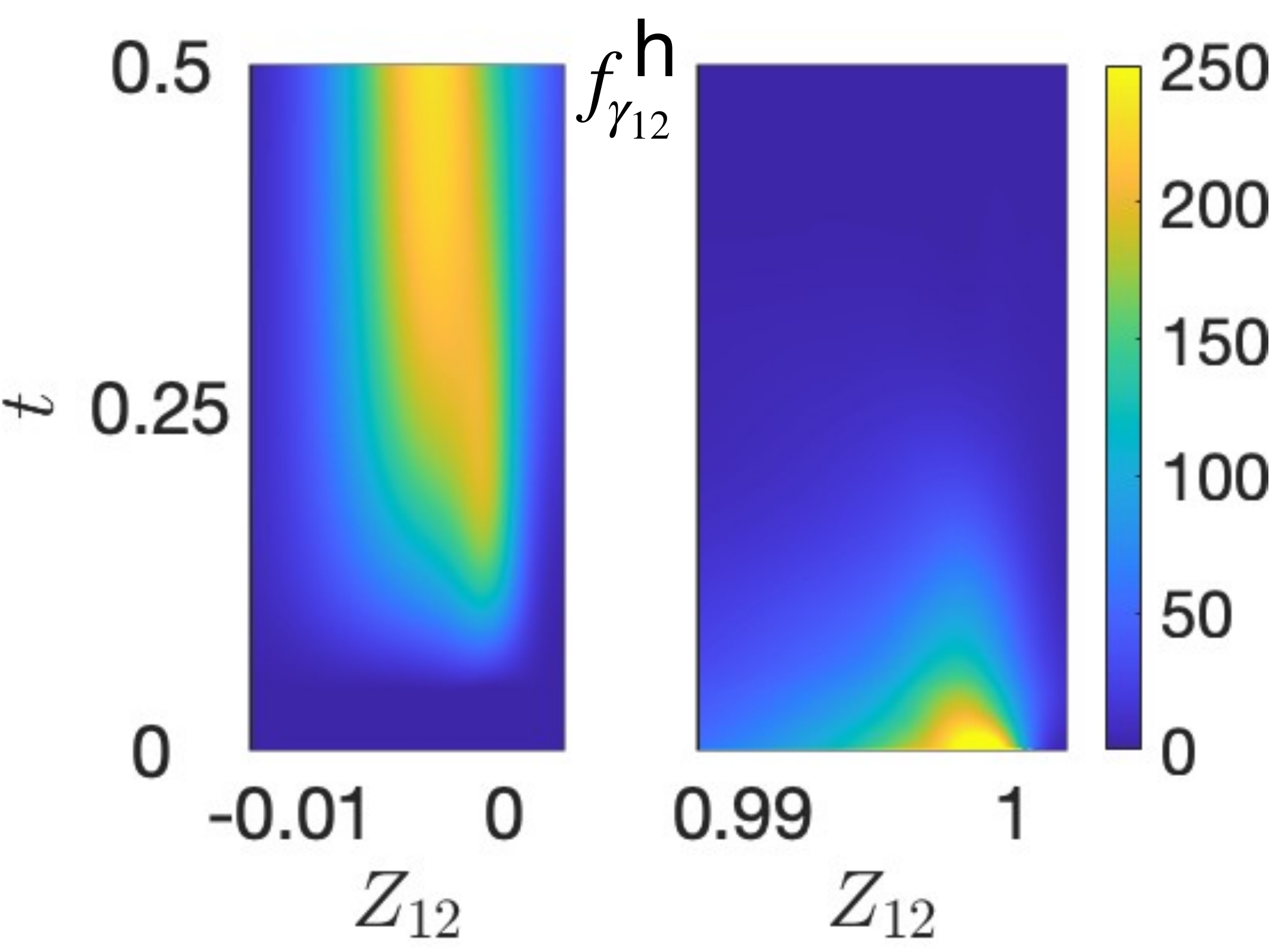}
	\end{subfigure} 
\caption{(Left) Learned regression function $\hat\rmc(Z_{12},t)$ estimated from $N_\text{MC}^\text{tr} = 2\times 10^3$ MC realizations of equation (\ref{eqh}) at sparse observation times $\mathbf T_\nu$ with $\nu=10^3$. The behavior of $\hat\rmc$ is difficult to distinguish using a single scale (bottom). Hence, we partition the domain about the (red dashed) line $Z_{12}\approx 0$ and plot the left and right sides on difference scales (top). (Middle) The partially separable advection coefficient based on the learned $\hat\rmc$. Likewise, the domain is split about the (gray dotted) line $Z_{12}\approx 0.95$. The left side is omitted on top as it closely resembles the bottom plot. (Right) Evolution of the homogeneous RoPDF $f_{\gamma_{12}}^\text{h}$ for $\gamma_{12}(t)$. The evolution is displayed only near the phase space boundaries since the middle portion of the domain always corresponds to relatively low probability states.}
\label{fig:pow_sol}
\end{figure}


\subsubsection{Error Analysis}
\label{sec:pow_err}

The temporal evolution of the solution $f_{\g_{12}}^\text{h}$ to the homogeneous equation \cref{gpdfeq} with estimated $\hat\rmc$ is given in \Cref{fig:pow_sol} (right) as well as snapshots at times $t=5\times10^{-3}$, $0.1$, and $0.4$ in \Cref{fig:powpdfs}. At $t=0$, $f_{\g_{12}}^\text{h}$ is unimodal and nearly symmetric around its original deterministic equilibrium point, indicating that the line is fully operational. Due to the line's low power rating together with stochastic fluctuations at the loads, as time evolves, the probability of transmission failure increases and quickly skews the density left. The small value of $\mathbf D_\g$ causes the RoPDF to transition extraordinarily fast to form a new mode near $Z_{12}\approx0$---its mass is approximately the line's failure probability at any given time. The scale at which this transience occurs in addition to the complexity of $\rmc$ from multiphysics is precisely why the accuracy of $f_{\g_{12}}^\text{h}$ degrades when data is temporally sparse. This was also true for the linear system in \Cref{sec:lin}; however, it is much more apparent in this application. Due to the nature of $\hat\rmc$ near zero (see \Cref{fig:transform} (right) and \Cref{fig:pow_sol} (left)), small shifts away from the true $\rmc$ can result in large changes in magnitude, which is exacerbated by interpolation. Moreover, $\rmc$ also rapidly oscillates in time (on small scales), which is not captured well by the estimate $\hat\rmc$ due to data sparsity. The combined difficulties introduce error to the homogeneous solution that compounds over time, as seen in the RoPDF snapshot \Cref{fig:powpdfs} (right) and the $L_1$ error evolution in \Cref{fig:pow_err} (middle).

Also seen in \Cref{fig:powpdfs}, the nudged and DNN observers perform well over the timecourse compared to the homogeneous solution. Both observers capture all qualitative aspects of the yardstick solution exceptionally well during early and middle times. At late times (right), both struggle to fully capture the right mode; however, this can be contributed to the vastly differing scales of the two modes. Note, these modes at $t=0.4$ (right) are displayed on different scales to emphasize this discrepancy at the right mode. Given that the mass of the right mode is only a fraction of the RoPDF's total mass, such discrepancies do not significantly influence either observer's error against the yardstick MC solution, as seen in \Cref{fig:pow_err} (middle). 

\begin{figure}[t!] 
	\centering
	\begin{subfigure}{.33\textwidth}
  		\centering
  		\includegraphics[width=1\linewidth]{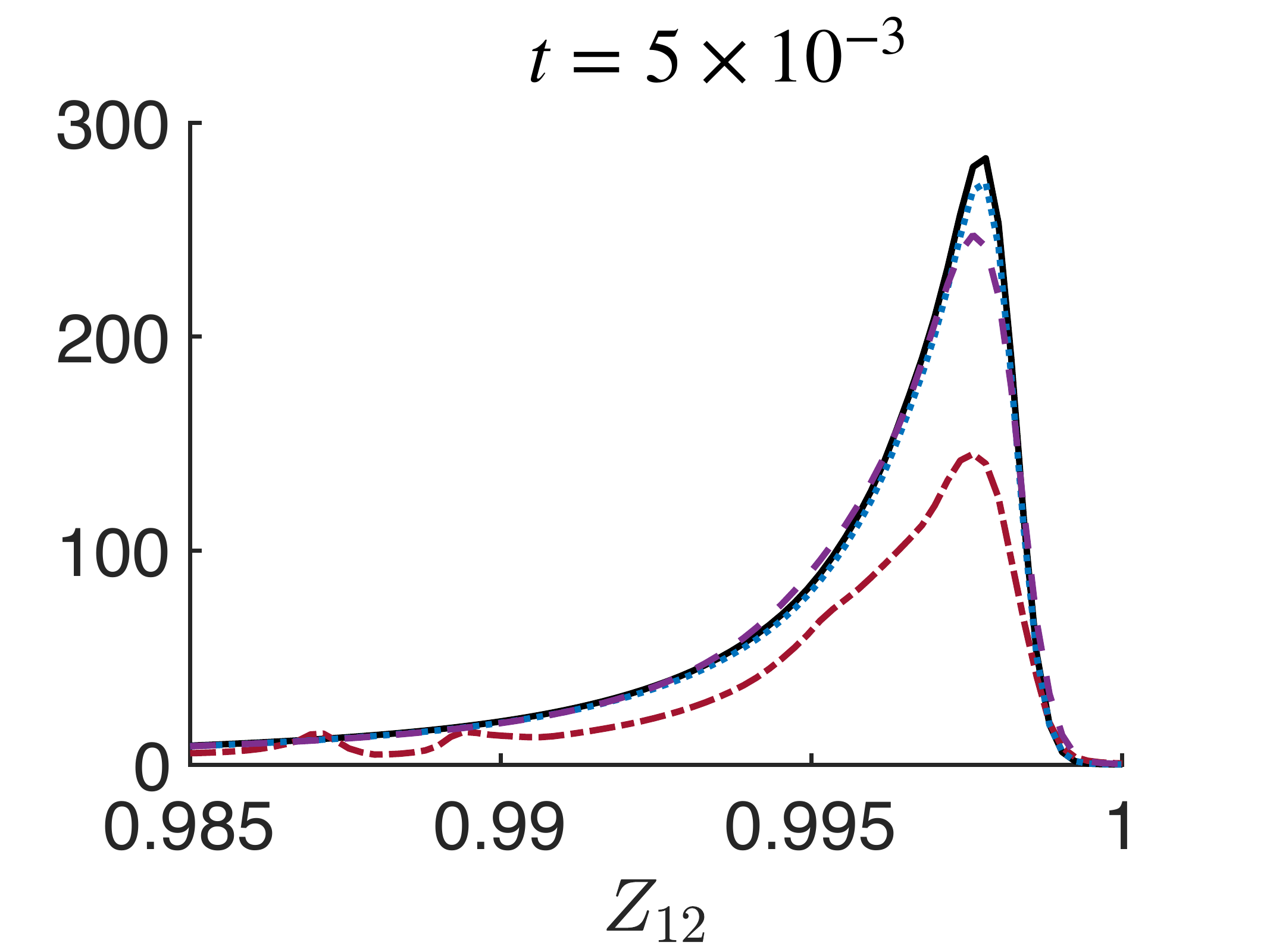}
	\end{subfigure}%
	\begin{subfigure}{.33\textwidth}
 	 	\centering
 		 \includegraphics[width=1\linewidth]{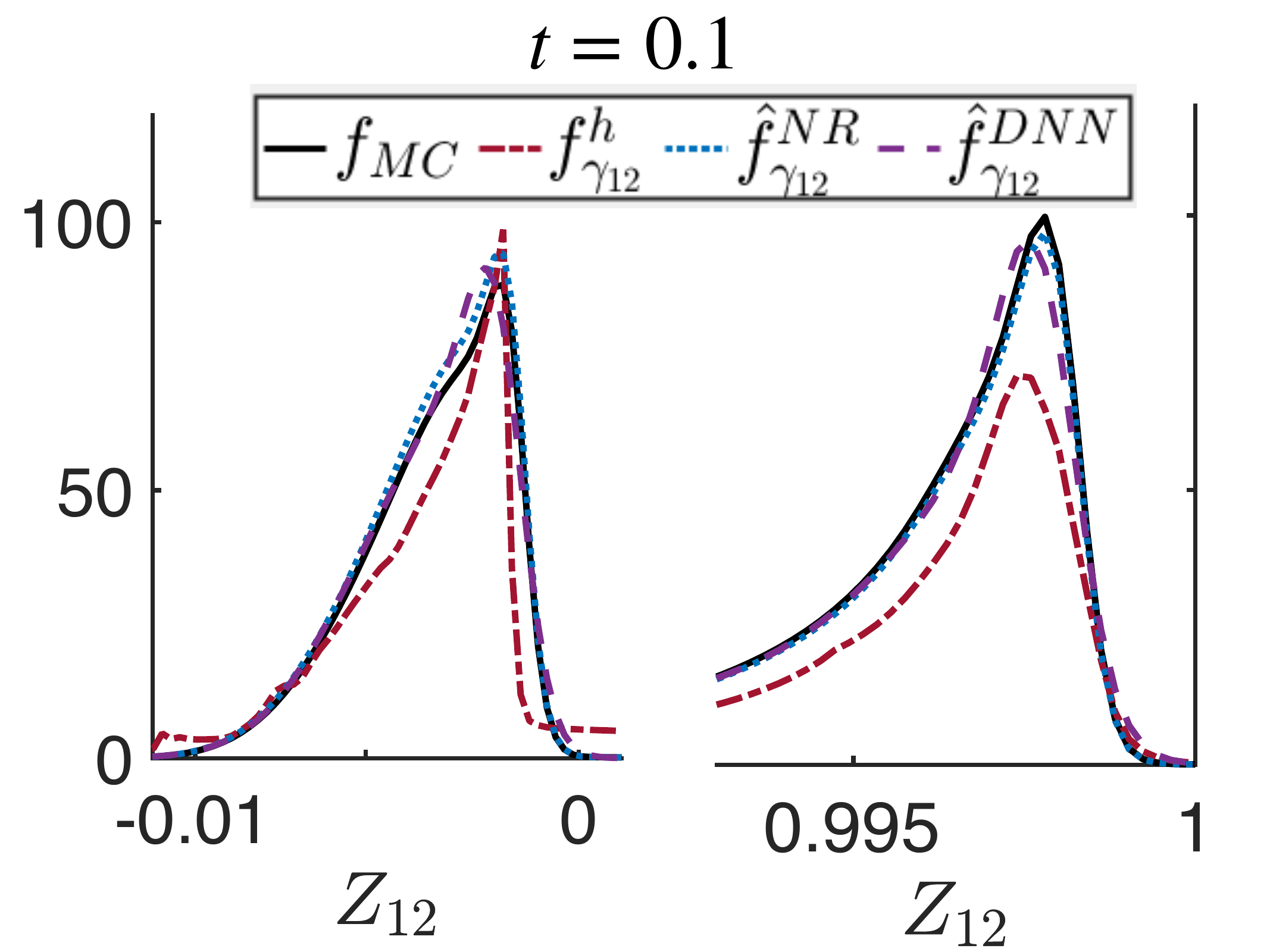}
	\end{subfigure}%
	\begin{subfigure}{.33\textwidth}
  		\centering
 	 \includegraphics[width=1\linewidth]{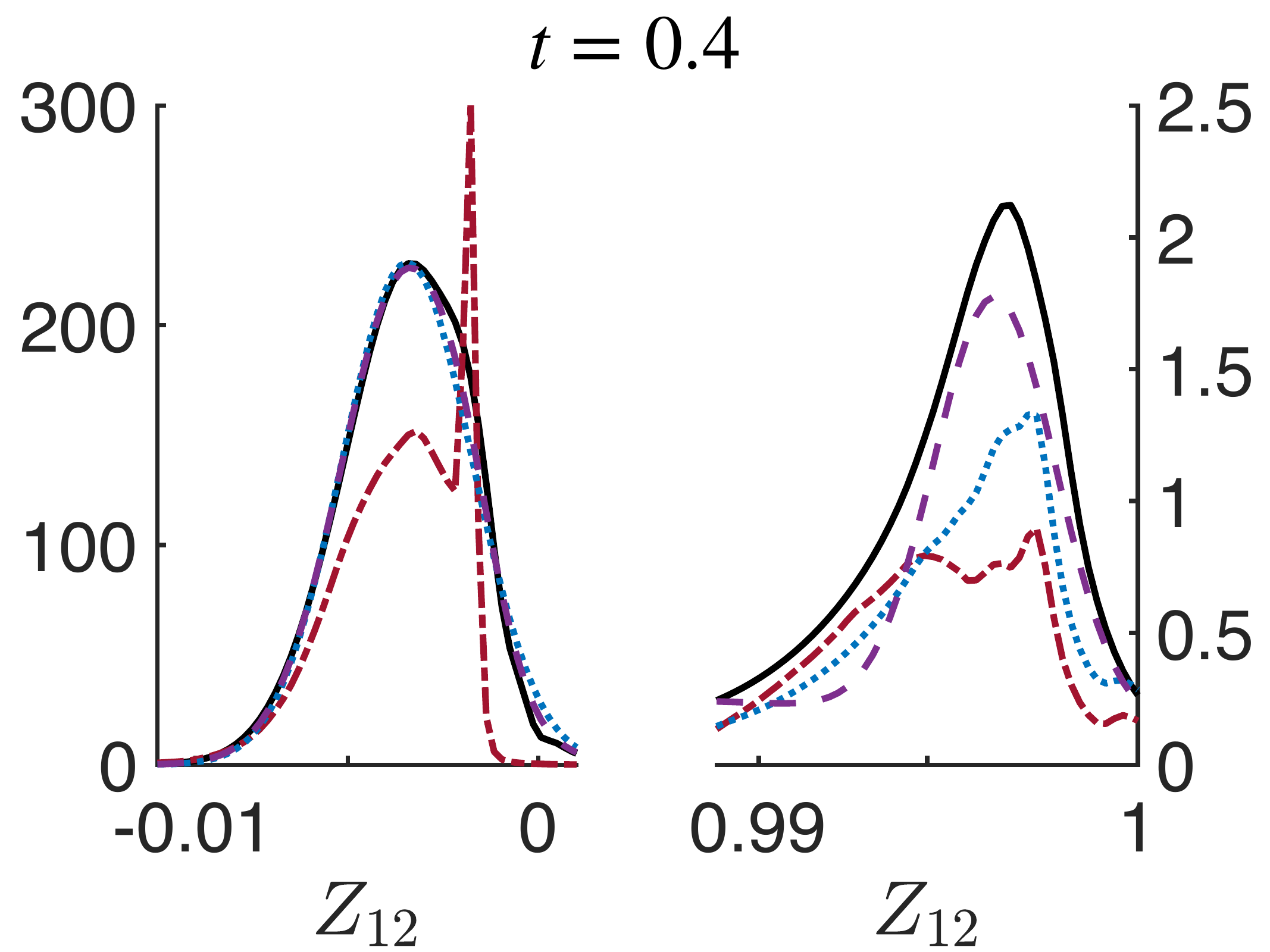}
	\end{subfigure} 
\caption{Comparison of the yardstick $f_\text{MC}$ (solid black), the homogeneous solution $f_{\gamma_{12}}^\text{h}$ (dot-dashed red), the nudged observer $\hat f_{\gamma_{12}}^\text{NR}$ (dotted blue), and the DNN observer $\hat f_{\gamma_{12}}^\text{DNN}$ (dashed purple) for $\nu=10^3$ and $N_\text{MC}^\text{tr} = 2\times 10^3$ at times $t=0.005$ (left), $0.1$ (middle), and $0.4$, (right). Similar to the PDFs in \Cref{fig:pow_sol}, the phase space is restricted near the boundaries to emphasize the bimodal behavior at the latter two times. Moreover, at $t = 0.4$ (right), the modes are given on two scales, revealing that the right mode shrinks, but does not vanish.} 
\label{fig:powpdfs}
\end{figure}

While small at late times, the right mode of $f_{\g_{12}}$ does not vanish, even if the final time $T_f$ is increased significantly. Hence, there is a nonzero, albeit small, probability that the line does not trip. If one desires a highly refined estimate of this probability, observer discrepancy at the RoPDF's right mode must be reduced. When the sparsity level $\nu$ is fixed, the only surefire way to improve the nudged observer is to increase $N_\text{MC}^\text{tr}$, i.e., the amount of training data. This is also true for the DNN observer, but to a lesser extent by means of normalization. Unlike the linear RODE in \Cref{sec:lin}, there is no clear-cut approach to normalizing the RoPDFs and observations locations for improved training. However, since the CDFs corresponding to the observations  $f_\text{MC}^{\text{tr},\nu}$ were computed for $\hat\rmc$ estimation, for each observation time, we multiply the spatial grid by these CDFs. We then shift, scale, and apply the square root transform to obtain new observation locations for training. This serves to disperse the essential supports of the RoPDF modes, making them easier to learn. In addition to applying square root transforms to the RoPDFs, we also shift and scale them, along with the observation times, so that the RoPDFs and spatiotemporal locations are all on the same scale. This approach to normalization considerably improves DNN estimates of the right mode at late times. Alternative approaches to normalization may yield better results, but none are perfect---the observer still depends on the quality of $f_\text{MC}^{\text{tr},\nu}$ and therefore $N_\text{MC}^\text{tr}$.

\Cref{fig:pow_err} (left) demonstrates the qualitative behavior to be expected of the learned defect solution $\hat f_{\g_{12}}^\text{d}$ associated with the DNN observer (dashed purple). The defect corresponding to the nudged observer (dotted blue), computed as $\hat f_{\gamma_{12}}^\text{NR} - f_{\gamma_{12}}^\text{h}$, is also plotted for reference. They are nonperiodic and exhibit steep gradients. As seen in the middle plot, the errors of $\hat f_{\g_{12}}^\text{NR}$ and $\hat f_{\g_{12}}^\text{DNN}$ against the yardstick solution are quite similar, with $\hat f_{\g_{12}}^\text{NR}$ performing slightly better at early times, i.e., during the initial transience. As previously discussed, DNN obsever error could possibly be improved to match that of nudging via alternative normalization. However, we contribute the nudged observer's better success to it's ability to dynamically overcome the highly nontrivial error distributions associated with these RoPDFs. We remark that, in this setting, the EnKF would likely perform poorly against the nudged observer due the errors being large and non-Gaussian.

\begin{figure}[t!] 
	\centering  
  	\includegraphics[width=.3285\linewidth]{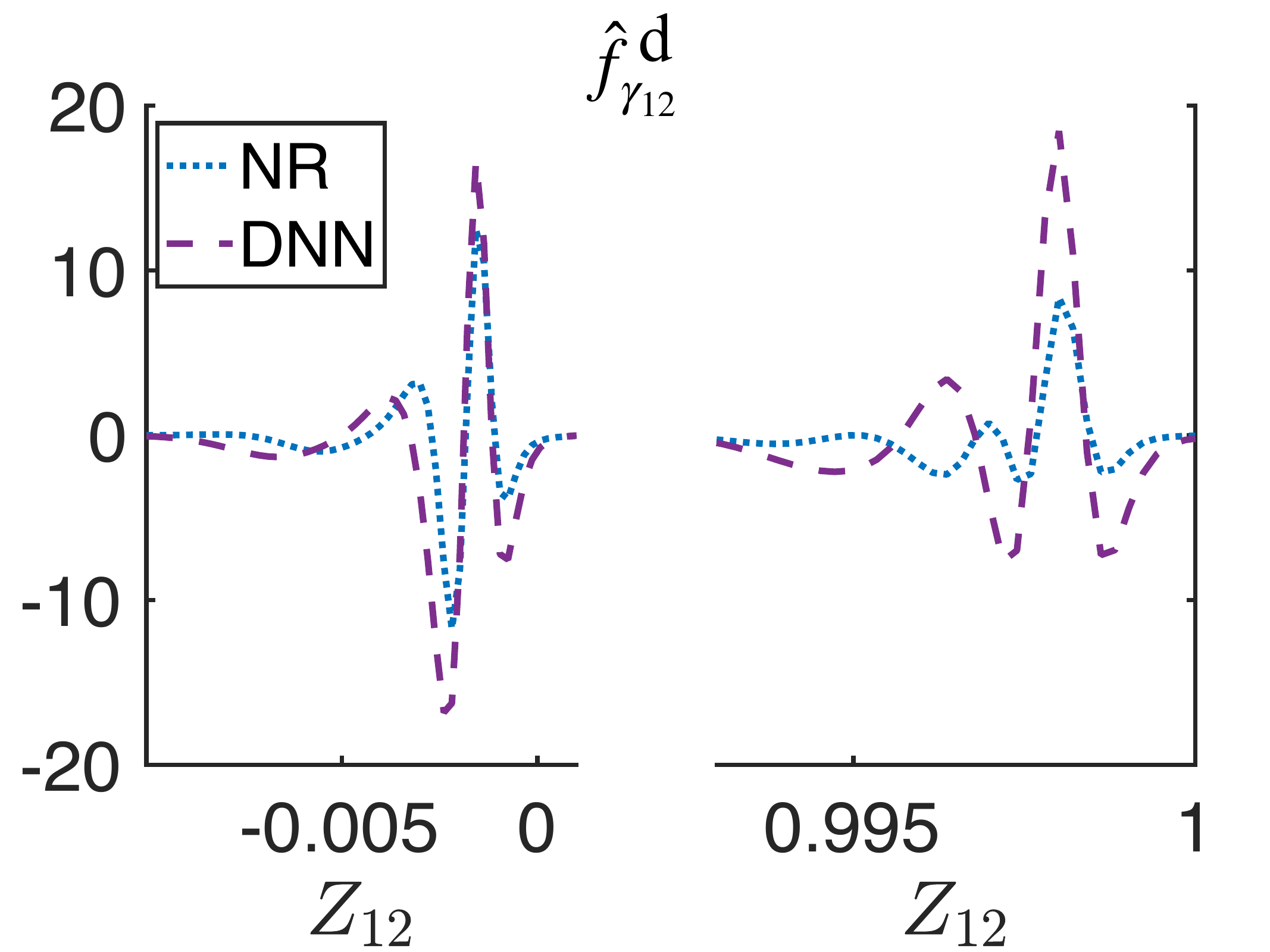}
    \includegraphics[width=.3285\linewidth]{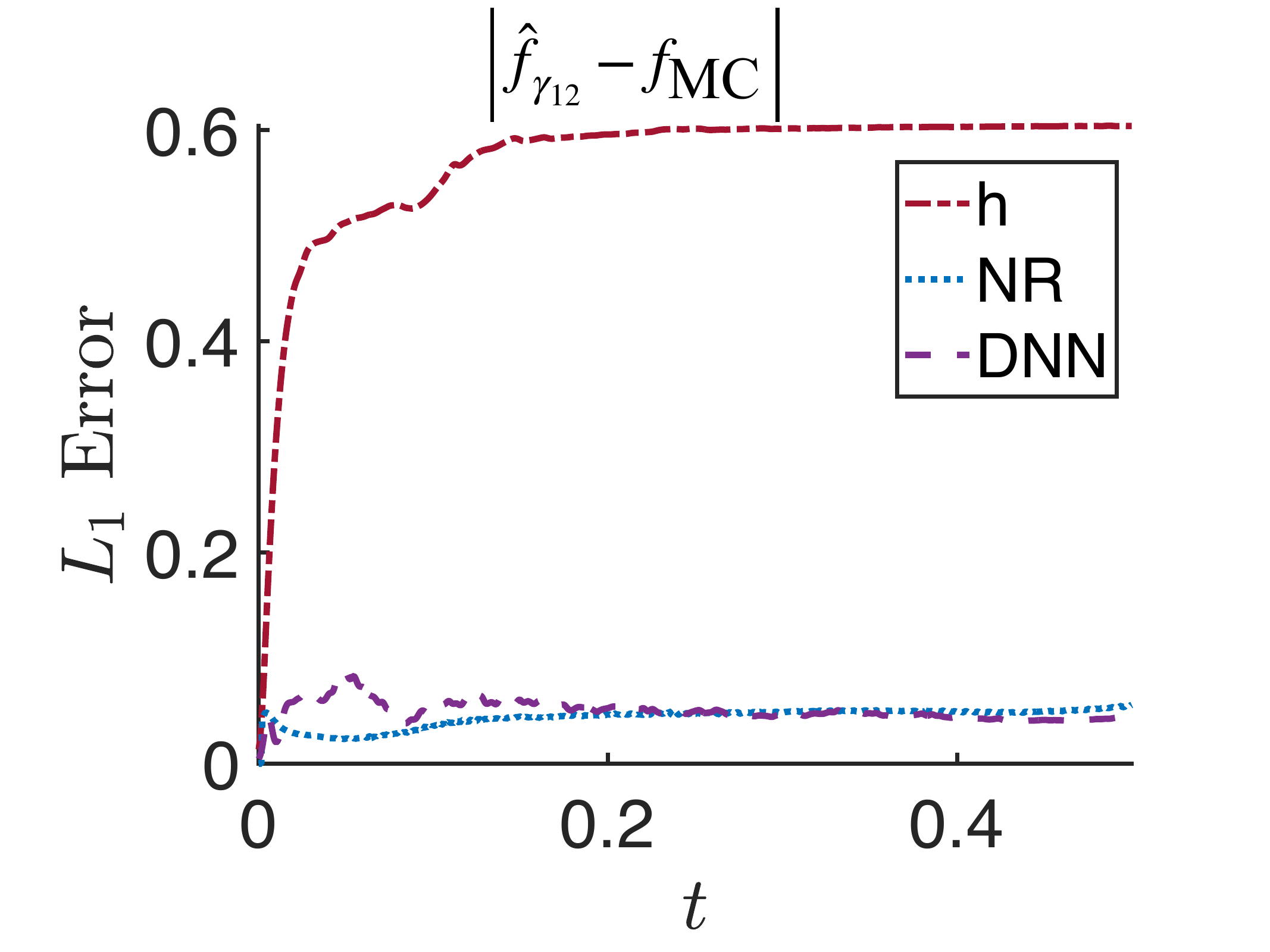}
    \includegraphics[width=.3285\linewidth]{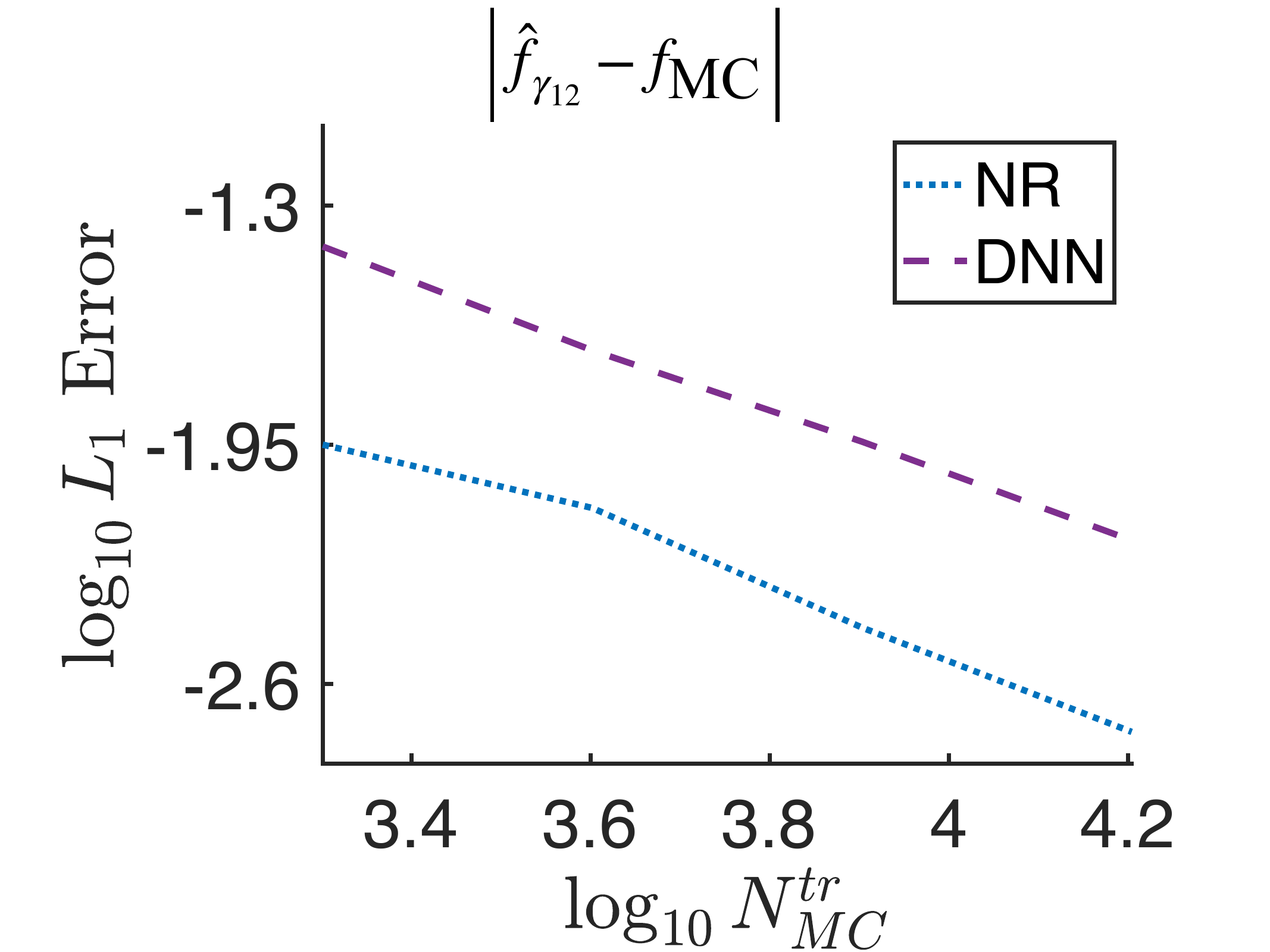}
\caption{(Left) The defect solutions $\hat{f}_{\gamma_{12}}^\text{d}$ for the DNN and nudged observers at time $t = 0.1$ with sparsity factor $\nu=10^3$ and $N_\text{MC}^\text{tr} = 2\times 10^3$ MC samples. The former is the (prediction of the) DNN \cref{mlp} while the latter is computed {\emph{ex post facto}} as $\hat f_{\gamma_{12}}^\text{NR} - f_{\gamma_{12}}^\text{h}$. (Middle) Temporal evolution of $L_1$ errors for the homogeneous solution, $f_{\gamma_{12}}^\text{h}$, the nudged observer $\hat f_{\gamma_{12}}^\text{NR}$, and the DNN observer $\hat f_{\gamma_{12}}^\text{DNN}$. (Right) Convergence rates, on a log-log scale, for the normalized spatiotemporal $L_1$ errors of the nudged and DNN observers as $N_\text{MC}^\text{tr}$ increases.}
\label{fig:pow_err}
\end{figure}

\Cref{fig:pow_err} (right) displays that even though the power systems dynamics are significantly more complex, stiff, and higher dimensional than the linear system from \Cref{sec:lin}, the RoPDF method obtains the same $\bo(1/N_\text{MC}^\text{tr})$ convergence as the number of MC realizations increases, demonstrating the method's robustness. Moreover, given that the yardstick MC solution requires $N_\text{MC}=10^5$ realizations of the stiff, $47$-dimensional RODE system (even with fast, adaptive KDE) and the RoPDF method needs fewer than $N_\text{MC}^\text{tr}=10^3$ realizations to achieve (less than) $1\%$ $L_1$ error, regardless of the assimilation approach, the method requires relatively few computational resources. For this application, the computational costs of numerically integrating the RoPDF equation, including the additional cost of DNN training, is but a small fraction of total costs. Therefore, comparing $N_\text{MC}^\text{tr}$ to $N_\text{MC}$ represents the computational speedup of the method sufficiently well. In particular, we see a speedup of at least two orders of magnitude (depending on desired error tolerance) of the RoPDF method compared to the MC approach.

\begin{remark} \label{rmk:mode}
Given the ``spikiness" of $f_{\g_{12}}$, i.e., its modes having small essential support, a RoCDF formulation is likely a better approach for the DNN observer. The CDFs have nicer regularity than the PDFs, making them easier to learn with less manual normalization/tuning. However, our presentation is limited to RoPDFs since the literature has largely focused on general PDF methods for RODEs and Langevin-type systems driven by colored noise. 
\end{remark}


\section{Conclusions}
\label{sec:conclusions} 

In this work, we have developed a physics-informed framework for studying uncertainty propagation of physical quantities of interest in high-dimensional and multiscale stochastic dynamical systems. In particular, we presented a derivation of an exact RoPDF equation and a regression-based approach to closures, enabling the characterization of full probabilistic profiles at all times with low computational complexity. Furthermore, we introduced two physics-informed data assimilation procedures to address issues arising in stiff systems, namely nudging/Newtonian relaxation and deep neural networks, which assimilates in low-fidelity observations at sparse observation times with negligible cost, improving density estimates. Finally, we showed the accuracy of our method on characterizing uncertainty in both a synthetic stiff linear system and an at-scale power system cascading failure model using IEEE case data. The results of our method demonstrate promising and practical uses in the prediction of complex stochastic phenomena. 

Several challenges and opportunities arise following this work. Firstly, convergence rates of the RoPDF method are dependent on three factors: (1) error from KDE, (2) truncation error from PDE scheme and (3) estimation error from regression functions. A precise characterization of solution accuracy can be analyzed from a learning theory perspective, particularly for the effect of noise distributions and overfitting. Secondly, we performed preliminary experimentation of using DNNs for the discovery of model defects. We believe that DNNs could have strong extrapolation power once trained with more sophisticated architectures. To name a few, Fourier-based DNNs are known to capture stiff dynamics well. Additionally, long short-term memory networks can be used to impose temporal ordering, which is more suitable for learning from (sparse) time-series observations. Particularly, when observations cease to exist, the design of reliable DNN extrapolation is necessary, which was beyond the scope of this study. Finally, the natural extension to uncertainty quantification for vector-valued QoIs is of great practical interest, such as rare-event probability estimations for multiple line failures in the power system model \cref{pham2}. However, the issue of dimensionality returns when the reduced state space itself is high-dimensional, which may potentially be resolved via structured low-complexity methods, such as tensor-networks, flow-based generative models, and/or a combination of such strategies where an initial product measure can be formed from marginal densities solved using the 1D RoPDF method, and then optimized to approach the correct reduced-order joint density. 


\bibliographystyle{siamplain}
\bibliography{references}

\begin{thebibliography}{10}

\bibitem{makima}
{\sc H.~Akima}, {\em A method of bivariate interpolation and smooth surface
  fitting based on local procedures}, Communications of the ACM, 17 (1974),
  pp.~18--20.

\bibitem{Alawadhi_2018}
{\sc A.~A. Alawadhi, F.~Boso, and D.~M. Tartakovsky}, {\em Method of
  distributions for water-hammer equations with uncertain parameters}, Water
  Resour. Res., 54 (2018), pp.~9398--9411,
  \url{https://doi.org/10.1029/2018WR023383}.

\bibitem{Asar}
{\sc O.~Asar, O.~Ilk, and O.~Dag}, {\em Estimating {B}ox-{C}ox power
  transformation parameter via goodness-of-fit tests}, Commun. Stat. Simul.
  Comput., 46 (2017), pp.~91--105,
  \url{https://doi.org/10.1080/03610918.2014.957839}.

\bibitem{Boso-2016-method}
{\sc F.~Boso and D.~M. Tartakovsky}, {\em The method of distributions for
  dispersive transport in porous media with uncertain hydraulic properties},
  Water Resour. Res., 52 (2016), pp.~4700--4712,
  \url{https://doi.org/10.1002/2016WR018745}.

\bibitem{Boso_2020}
{\sc F.~Boso and D.~M. Tartakovsky}, {\em Data-informed method of distributions
  for hyperbolic conservation laws}, {SIAM} J. Sci. Comput., 42 (2020),
  pp.~A559--A583, \url{https://doi.org/10.1137/19m1260773}.

\bibitem{Botev_2010}
{\sc Z.~I. Botev, J.~F. Grotowski, and D.~P. Kroese}, {\em Kernel density
  estimation via diffusion}, Ann. Stat., 38 (2010),
  \url{https://doi.org/10.1214/10-aos799}.

\bibitem{Boulanger_2015}
{\sc A.-C. Boulanger, P.~Moireau, B.~Perthame, and J.~Sainte-Marie}, {\em Data
  assimilation for hyperbolic conservation laws: A {L}uenberger observer
  approach based on a kinetic description}, Commun. in Math. Sci., 13 (2015),
  pp.~587--622, \url{https://doi.org/10.4310/cms.2015.v13.n3.a1}.

\bibitem{bover_1978}
{\sc D.~C.~C. Bover}, {\em Moment equation methods for nonlinear stochastic
  systems}, J. Math. Anal. Appl., 65 (1978), pp.~306--320,
  \url{https://doi.org/10.1016/0022-247X(78)90182-8}.

\bibitem{Bowman1997}
{\sc A.~W. Bowman and A.~Azzalini}, {\em Applied Smoothing Techniques for Data
  Analysis}, Clarendon Press, 1997.

\bibitem{BRENNAN2018281}
{\sc C.~Brennan and D.~Venturi}, {\em Data-driven closures for stochastic
  dynamical systems}, J. Comput. Phys., 372 (2018), pp.~281--298,
  \url{https://doi.org/10.1016/j.jcp.2018.06.038}.

\bibitem{Chorin_2002}
{\sc A.~J. Chorin, O.~H. Hald, and R.~Kupferman}, {\em Optimal prediction with
  memory}, Physica D: Nonlin. Phenom., 166 (2002), pp.~239--257,
  \url{https://doi.org/10.1016/S0167-2789(02)00446-3}.

\bibitem{DeMarco_1987}
{\sc C.~L. DeMarco and A.~Bergen}, {\em A security measure for random load
  disturbances in nonlinear power system models}, IEEE Trans. Circuits Syst.,
  34 (1987), pp.~1546--1557, \url{https://doi.org/10.1109/TCS.1987.1086092}.

\bibitem{evans}
{\sc L.~C. Evans}, {\em Partial Differential Equations}, AMS, Providence,
  2nd~ed., 2010, \url{https://doi.org/10.1090/gsm/019}.

\bibitem{Fu_2020}
{\sc X.~Fu, L.~Chang, and D.~Xiu}, {\em Learning reduced systems via deep
  neural networks with memory}, J. Mach. Learn. Model. Comput., 1 (2020),
  pp.~97--118, \url{https://doi.org/10.1615/.2020034232}.

\bibitem{Giles}
{\sc M.~B. Giles}, {\em Multilevel {Monte Carlo} path simulation}, Oper. Res.,
  56 (2008), pp.~607--617, \url{https://doi.org/10.1287/opre.1070.0496}.

\bibitem{Han2017}
{\sc X.~Han and P.~E. Kloeden}, {\em Random Ordinary Differential Equations},
  Springer Singapore, Singapore, 2017, pp.~15--27,
  \url{https://doi.org/10.1007/978-981-10-6265-0_2}.

\bibitem{Hastie_2009}
{\sc T.~Hastie, R.~Tibshirani, and J.~Friedman}, {\em The Elements of
  Statistical Learning}, Springer New York, 2nd~ed., 2009,
  \url{https://doi.org/10.1007/978-0-387-84858-7}.

\bibitem{jazwinski1970}
{\sc A.~H. Jazwinski}, {\em Stochastic Processes and Filtering Theory}, Dover,
  Mineola, NY, 1970.

\bibitem{Kraichnan_1961}
{\sc R.~H. Kraichnan}, {\em Dynamics of nonlinear stochastic systems}, J. of
  Math. Phys., 2 (1961), pp.~124--148, \url{https://doi.org/10.1063/1.1724206}.

\bibitem{NEURIPS2021_df438e52}
{\sc A.~Krishnapriyan, A.~Gholami, S.~Zhe, R.~Kirby, and M.~W. Mahoney}, {\em
  Characterizing possible failure modes in physics-informed neural networks},
  in Adv. Neural Inf. Process. Syst., vol.~34, 2021, pp.~26548--26560.

\bibitem{Lagergren2_2020}
{\sc J.~H. Lagergren, J.~T. Nardini, R.~E. Baker, M.~J. Simpson, and K.~B.
  Flores}, {\em Biologically-informed neural networks guide mechanistic
  modeling from sparse experimental data}, PLoS Comput. Biol., 16 (2020),
  p.~e1008462, \url{https://doi.org/10.1371/journal.pcbi.1008462}.

\bibitem{Lakshmivarahan_2013}
{\sc S.~Lakshmivarahan and J.~M. Lewis}, {\em Nudging methods: A critical
  overview}, in Data Assimilation for Atmospheric, Oceanic and Hydrologic
  Applications (Vol. {II}), Springer Berlin Heidelberg, 2013, pp.~27--57,
  \url{https://doi.org/10.1007/978-3-642-35088-7_2}.

\bibitem{Dimet_2010}
{\sc F.-X. Le~Dimet and O.~Talagrand}, {\em Variational algorithms for analysis
  and assimilation of meteorological observations: theoretical aspects}, Tellus
  A., 38 (2010), pp.~97--110,
  \url{https://doi.org/10.1111/j.1600-0870.1986.tb00459.x}.

\bibitem{Lei_2015}
{\sc L.~Lei and J.~P. Hacker}, {\em Nudging, ensemble, and nudging ensembles
  for data assimilation in the presence of model error}, Mon. Weather Rev., 143
  (2015), pp.~2600--2610, \url{https://doi.org/10.1175/mwr-d-14-00295.1}.

\bibitem{Li_2022}
{\sc J.~Li and P.~Stinis}, {\em Model reduction for a power grid model}, J.
  Comput. Dyn., 9 (2022), pp.~1--26, \url{https://doi.org/10.3934/jcd.2021019}.

\bibitem{maltbaphd}
{\sc T.~E. Maltba}, {\em The Method of Distributions for Random Ordinary
  Differential Equations}, PhD thesis, UC Berkeley, 2023,
  \url{https://www.proquest.com/dissertations-theses/method-distributions-random-ordinary-differential/docview/2867938203/se-2}.

\bibitem{Maltba}
{\sc T.~E. Maltba, P.~A. Gremaud, and D.~M. Tartakovsky}, {\em Nonlocal {PDF}
  methods for {L}angevin equations with colored noise}, J. Comput. Phys., 367
  (2018), pp.~87--101, \url{https://doi.org/10.1016/j.jcp.2018.04.023}.

\bibitem{MaltbaPES}
{\sc T.~E. Maltba, V.~Rao, and D.~A. Maldonado}, {\em Learning the evolution of
  correlated stochastic power system dynamics}, in 2022 IEEE Power \& Energy
  Society General Meeting (PESGM), 2022, pp.~01--05,
  \url{https://doi.org/10.1109/PESGM48719.2022.9916982}.

\bibitem{Maltba2022}
{\sc T.~E. Maltba, H.~Zhao, and D.~M. Tartakovsky}, {\em Autonomous learning of
  nonlocal stochastic neuron dynamics}, Cogn. Neurodyn., 16 (2022),
  pp.~683--705, \url{https://doi.org/10.1007/s11571-021-09731-9}.

\bibitem{pytorch}
{\sc A.~Paszke, S.~Gross, F.~Massa, A.~Lerer, J.~Bradbury, G.~Chanan,
  T.~Killeen, Z.~Lin, N.~Gimelshein, L.~Antiga, A.~Desmaison, A.~Kopf, E.~Yang,
  Z.~DeVito, M.~Raison, A.~Tejani, S.~Chilamkurthy, B.~Steiner, L.~Fang,
  J.~Bai, and S.~Chintala}, {\em Pytorch: an imperative style, high-performance
  deep learning library}, in Advances in Neural Information Processing Systems
  32, 2019, pp.~8024--8035, \url{https://pytorch.org}.

\bibitem{RAISSI2019686}
{\sc M.~Raissi, P.~Perdikaris, and G.~E. Karniadakis}, {\em Physics-informed
  neural networks: A deep learning framework for solving forward and inverse
  problems involving nonlinear partial differential equations}, J. Comput.
  Phys., 378 (2019), pp.~686--707,
  \url{https://doi.org/10.1016/j.jcp.2018.10.045}.

\bibitem{rodgers2021adaptive}
{\sc A.~Rodgers, A.~Dektor, and D.~Venturi}, {\em Adaptive integration of
  nonlinear evolution equations on tensor manifolds}, J. of Sci. Comput., 92
  (2022), p.~39, \url{https://doi.org/10.1007/s10915-022-01868-x}.

\bibitem{Roth_2021}
{\sc J.~Roth, D.~A. Barajas-Solano, P.~Stinis, J.~Weare, and M.~Anitescu}, {\em
  A kinetic {M}onte {C}arlo approach for simulating cascading transmission line
  failure}, SIAM Multiscale Model. Simul., 19 (2021), pp.~208--241,
  \url{https://doi.org/10.1137/19m1306865}.

\bibitem{Tartakovsky-2015-Method}
{\sc D.~M. Tartakovsky and P.~A. Gremaud}, {\em Method of distributions for
  uncertainty quantification}, in Handbook of Uncertainty Quantification, e.~a.
  R.~Ghanem, ed., Springer, 2015, pp.~1--22,
  \url{https://doi.org/10.1007/978-3-319-11259-6_27-1}.

\bibitem{TAVERNIERS2020}
{\sc S.~Taverniers and D.~M. Tartakovsky}, {\em Estimation of distributions via
  multilevel {M}onte {C}arlo with stratified sampling}, J. Comput. Phys., 419
  (2020), p.~109572, \url{https://doi.org/10.1016/j.jcp.2020.109572}.

\bibitem{Venturi-2012-A}
{\sc D.~Venturi, T.~P. Sapsis, H.~Cho, and G.~E. Karniadakis}, {\em A
  computable evolution equation for the joint response-excitation probability
  density function of stochastic dynamical systems}, Proc. R. Soc. A, 468
  (2012), pp.~759--783, \url{https://doi.org/10.1098/rspa.2011.0186}.

\bibitem{Wang-2013-Probability}
{\sc P.~Wang, A.~M. Tartakovsky, and D.~M. Tartakovsky}, {\em Probability
  density function method for {L}angevin equations with colored noise}, Phys.
  Rev. Lett., 110 (2013), p.~140602, \url{https://doi.org/10.1103/PhysRev
  Lett.110.140602}.

\bibitem{Wang_2013}
{\sc P.~Wang, D.~M. Tartakovsky, K.~D. Jarman, and A.~M. Tartakovsky}, {\em
  {CDF} solutions of {B}uckley--{L}everett equation with uncertain parameters},
  SIAM Multiscale Model. Simul., 11 (2013), pp.~118--133,
  \url{https://doi.org/10.1137/120865574}.

\bibitem{Zheng_2015phd}
{\sc H.~Zheng}, {\em Lyapunov approaches to power system cascading failure
  analysis}, PhD thesis, University of Wisconsin, Madison, 2015.

\bibitem{Zheng3}
{\sc H.~Zheng and C.~L. DeMarco}, {\em A new dynamic performance model of motor
  stalling and fidvr for smart grid monitoring/planning}, IEEE Trans. Smart
  Grid, 7 (2016), pp.~1989--1996,
  \url{https://doi.org/10.1109/TSG.2016.2548082}.

\bibitem{Zhu_2021}
{\sc Y.~Zhu and D.~Venturi}, {\em Hypoellipticity and the {M}ori-{Z}wanzig
  formulation of stochastic differential equations}, J. Math. Phys., 62 (2021),
  p.~103505, \url{https://doi.org/10.1063/5.0035459}.

\bibitem{Zimmerman2011}
{\sc R.~D. Zimmerman, C.~E. Murillo-Sanchez, and R.~J. Thomas}, {\em
  {MATPOWER}: steady-state operations, planning, and analysis tools for power
  systems research and education}, {IEEE} Trans. Power Syst., 26 (2011),
  pp.~12--19, \url{https://doi.org/10.1109/tpwrs.2010.2051168}.

\end{thebibliography}


\appendix

\section{Nudging Convergence}
\label{app:nr}

For general random hyperbolic conservation laws, the method of distributions is formulated in a fashion similar to that of \Cref{thm} and \cite[Eq. 3]{Maltba2022} for RoPDF and joint PDF equations (respectively) corresponding to the RODE \cref{ivp2}. The kinetic description of the hyperbolic system is precisely the deterministic equation for the ``raw PDF'' $\Pi$. However, when the governing random PDE exhibits shocks, the method of distributions for $\Pi$ breaks down at singularities. This can be overcome by partitioning the domain and tracking shocks analytically, which was done in \cite{Alawadhi_2018,Wang_2013} for the water-hammer and Buckley-Leverett equations, respectively. However, analytically tracking shocks is rarely possible for general nonlinear hyperbolic PDEs. Instead, the kinetic defect term/collision operator $\M$ may be introduced as a source function in the raw PDF equation, incorporating all information regarding discontinuities. When the hyperbolic system exhibits smooth solutions, $\M$ is unique and identically zero. Otherwise, it can be written as the partial derivative of what is known as the kinetic entropy defect measure---it is exact, albeit generally unknown {\emph{a priori}}. Learning this defect in the CDF equations of nonlinear scalar conservation laws with random initial data was the focus of \cite{Boso_2020}, which largely motivated our extension to the setting of reduced-order equations. 

It was shown in \cite{Boulanger_2015} that nudging hyperbolic conservation laws at the kinetic level does not perturb the stability of the macroscopic system, which is beneficial for establishing strong convergence. In our setting of RoPDF equations, the conservation law \cref{ro-clpi} for $\Pi_{x_k}$ is a kinetic description and is exact. Therefore, the defect $\M$ vanishes and (scalar) nudging ensures that the corresponding observer $\hat\Pi_{x_k}$ converges globally and exponentially in $L_1$ to $\Pi_{x_k}$ with rate $\lambda>0$ when $H(\Pi_{x_k}) \equiv \Pi_{x_k}$, i.e., when observations are complete and exact. By virtue of the triangle inequality, 
\begin{equation}
    \lvert\lvert\hat f_{x_k} - f_{x_k}\rvert\rvert_1 
	\le \big\langle\lvert\lvert \hat \Pi_{x_k} - \Pi_{x_k} \rvert\rvert_1\big\rangle,
\end{equation}
the nudged observer $\hat f_{x_k}$ corresponding to the exact RoPDF equation enjoys the same global convergence to $f_{x_k}$, i.e., the solution to \cref{RoPDFsep}. In the case of temporally discrete observations, i.e., when the observations are sparse in time and complete in space, global convergence cannot be obtained. However, if the observations are interpolated over finite time intervals of length $T_w>0$ via the correction term
\[
	\lambda \sum_{l\in I} \phi_{T_w}(t-t_{m_l})\left(\Pi_{x_k}(X_k,t_{m_l})-\hat\Pi_{x_k}(X_k,t_{m_l})\right),
\] 
then, for any given time $T>0$, the nudged kinetic observer has bounded $L_1$ convergence: 
\begin{equation} \label{eq:l1con}
\lvert\lvert \hat \Pi_{x_k}(X_k,T) - \Pi_{x_k}(X_k,T) \rvert\rvert_1 \le C_0\text{e}^{\lambda L} + T_w\mathcal{I}(T),
\end{equation}
where $L$ is the number of time steps in $[0,T-T_w]$, $C_0$ is a constant depending on the initial condition, and $\mathcal{I}$ is a convergence-rate dependent quantity. Taking the ensemble mean provides a $\lambda$-dependent convergence bound for nudging the exact RoPDF equation \cref{RoPDFsep}. When the observations of $\Pi_{x_k}$ are noisy, under sufficient regularity conditions, one obtains a strong upper bound on the observation error in a homogeneous Sobolev norm and an optimal (scalar) nudging coefficient $\lambda>0$ (see \cite{Boulanger_2015}). 

In our nudging formulation, we have replaced the conditional expectations $\rmc_i$ with smooth estimators $\hat\rmc_i$, meaning that the kinetic defect term does not vanish. For exact but temporally sparse observations, this amounts to adding the term $\sup_{0<t\le T}\lvert\lvert \mathcal M(X_k,t)\rvert\rvert_1/\lambda$ to the bound in \cref{eq:l1con}. However, the more concerning issue is that we have introduced observation noise at the macroscopic level (via the low-fidelity estimates $f_\text{MC}^{\text{tr},\nu}$) rather than the kinetic level. Moreover, we are not aware of any existing literature that has addressed theoretical convergence in this setting. Since measurement noise often occurs on the macroscopic level in many applications, such results would be of great interest, and are indeed the focus of an ongoing body of work. In the meantime, we rely on the mounting empirical evidence for the practical nudging of ODEs/PDEs, including the new results for RoPDF equations in \Cref{sec:experiments}. 


\end{document}